%% file: gammaconv.tex
\documentclass{amsart}

\usepackage{graphicx,color}

\makeatletter
\@addtoreset{equation}{section}

\makeatother

\newtheorem{theorem}{Theorem}[section]
\newtheorem{proposition}[theorem]{Proposition}
\newtheorem{corollary}[theorem]{Corollary}
\newtheorem{lemma}[theorem]{Lemma}
\theoremstyle{definition}
\newtheorem{definition}[theorem]{Definition}
\newtheorem{remark}[theorem]{Remark}
\newtheorem{assumption}[theorem]{Assumption}
\newcommand{\Minimize}{\mathop{\rm Minimize}\limits}

\begin{document}

\title[Singular perturbation for adhesive obstacle problem]{Singular perturbation by bending for an adhesive obstacle problem}
\author{Tatsuya Miura}
\address{Graduate School of Mathematical Sciences, University of Tokyo, 3-8-1 Komaba, Meguro,
Tokyo, 153-8914 Japan}
\email{miura@ms.u-tokyo.ac.jp}

\begin{abstract}
A free boundary problem arising from materials science is studied in the one-dimensional case. The problem studied here is an obstacle problem for the non-convex energy consisting of a bending energy, tension and an adhesion energy. If the bending energy, which is a higher order term, is deleted then ``edge'' singularities of the solutions (surfaces) may occur at the free boundary as Alt-Caffarelli type variational problems. The main result of this paper is to give a singular limit of the energy utilizing the notion of $\Gamma$-convergence, when the bending energy can be regarded as a perturbation. This singular limit energy only depends on the state of surfaces at the free boundary as seen in singular perturbations for phase transition models.
\end{abstract}

\maketitle

\section{Introduction}

\subsection{Model and main results}

Let us consider a non-convex higher order variational problem in the one-dimensional case, which is the obstacle problem for the energy as proposed in \cite{PL08}:
\begin{eqnarray}\label{problem1}
	\Minimize_{u\geq\psi}:\ E_\varepsilon[u]=\varepsilon^2 \int \kappa^2 ds + \int ds - \int_{\{u=\psi\}} (1-\alpha)\ ds.
\end{eqnarray}
\noindent
Here a smooth function $\psi:[a,b]\rightarrow\mathbb{R}$ called an {\it obstacle} (function), a constant coefficient $\varepsilon>0$, and a continuous function $\alpha:[a,b]\rightarrow(0,1)$ are given.
The function $u$ is an admissible function constrained above the obstacle, and $\kappa$, $s$ denote the curvature and the arclength of the graph of $u$ respectively. 
The first term of the energy is called {\it bending energy}, the second term {\it tension}, the third term {\it adhesion energy}, and $\alpha$ is called {\it adhesion coefficient}. 
The multiple constant of the tension is normalized to one. 
The adhesion coefficient $\alpha$ can be inhomogeneous so that it may depend on a space variable. 
According to \cite{PL08}, this problem is motivated to determine the shape of membranes, interfaces or filaments on rippled surfaces (as Figure \ref{membranes}) in certain mesoscopic or nearly mesoscopic settings.
The coefficients $\varepsilon$ and $\alpha$ and the obstacle function $\psi$ depend on the setting of materials, scaling and so on. 
In this case the graph of $u$ is a membrane. 
In this paper we consider the one dimensional model so when one considers a membrane it depends on only one direction and invariant in other direction in our setting. 
In addition, we regard the bending energy which is a higher order term as a perturbation, that is, only consider for sufficiently small $\varepsilon>0$.

\begin{figure}[htbp]
	\begin{tabular}{cc}
		\begin{minipage}{0.45\hsize}
			\begin{center}
				\includegraphics[width=50mm]{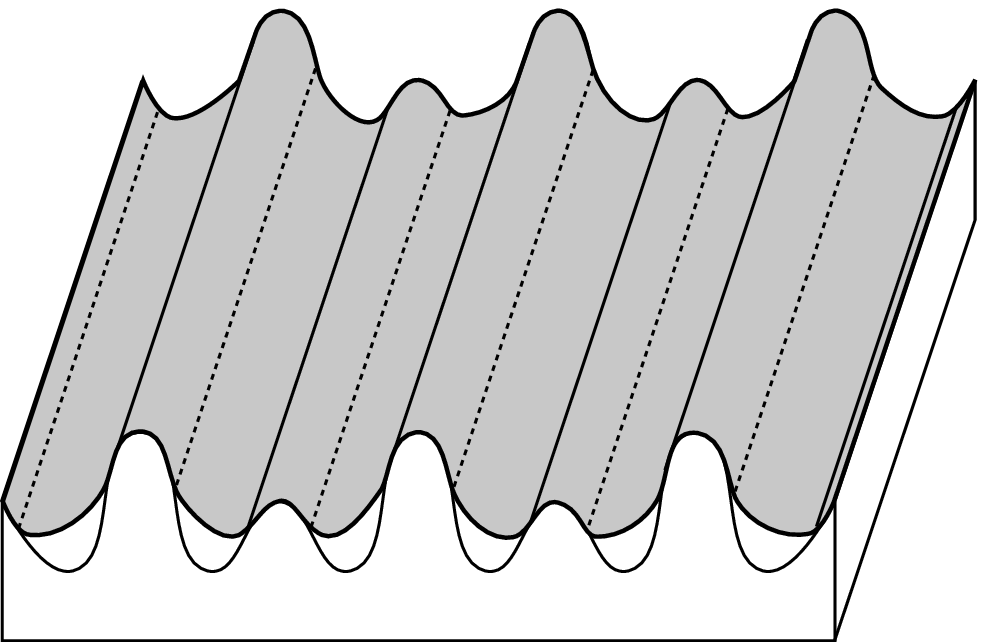}
			\end{center}
		\end{minipage}
		\begin{minipage}{0.45\hsize}
			\begin{center}
				\includegraphics[width=50mm]{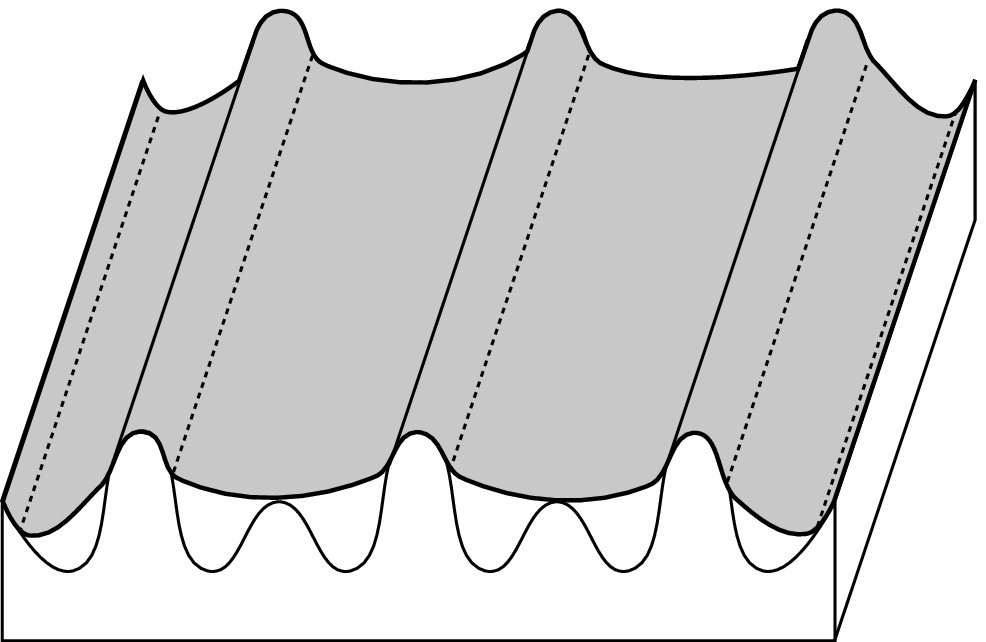}
			\end{center}
		\end{minipage}
	\end{tabular}
	\caption{membranes on rippled surface \label{membranes}}
\end{figure}

A characteristic point of this energy is to contain the adhesion energy. 
By this energy, surfaces shall reasonably adhere to the obstacle in order to decrease the energy. 
Accordingly, there may occur patterns as drawn in Figure \ref{membranes} in this model. 
However that law is complicated. 
One of physical and mathematical concerns is to perceive such pattern formation.

We shall first take the simplest approximation $\varepsilon=0$ in order to consider the case $\varepsilon$ is sufficiently small. 
This approximation simplifies the problem (\ref{problem1}) thus we can obtain many fine properties for minimizers (this is one of the important results of this paper, see Theorem \ref{regezero}). 
It is rigorously stated in Theorem \ref{regezero}, but roughly speaking the shape of any minimizer of $E_0$ is as drawn in Figure \ref{membrane}. 
There occur ``edge'' singularities at the free boundary and their angles are determined by the adhesion coefficient at the place of the free boundary, symbolically $\cos\theta=\alpha$ (Young's equation). 
This condition has been formally given in \cite{PL08}.

However, even if $\varepsilon=0$, the energy is not convex and may admit that there exist multiple minimizers lacking consistency in their shape. 
That is to say, for instance, either of two different states as drawn in Figure \ref{membrane} might be a minimizer for the same energy. 
Since this minimizing problem is considered to be a physical model, this non-uniqueness may be due to some small effect, perhaps, of higher order terms. 
Therefore, restoring the effect of the bending energy, it is expected to ameliorate the approximation.

\begin{figure}[htbp]
	\begin{tabular}{cc}
		\begin{minipage}{0.45\hsize}
			\begin{center}
				\includegraphics[width=50mm]{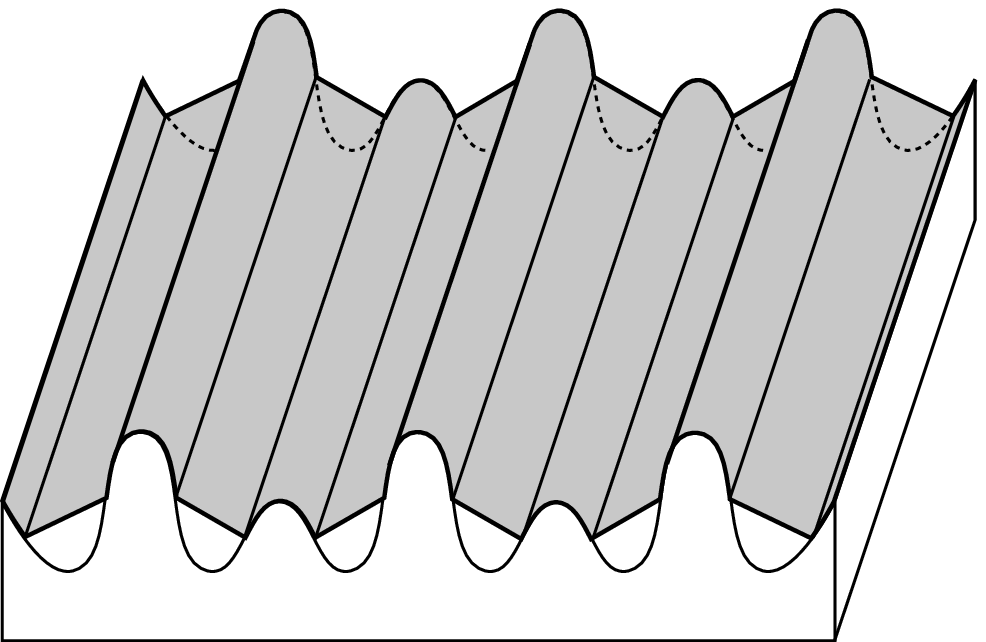}
			\end{center}
		\end{minipage}
		\begin{minipage}{0.45\hsize}
			\begin{center}
				\includegraphics[width=50mm]{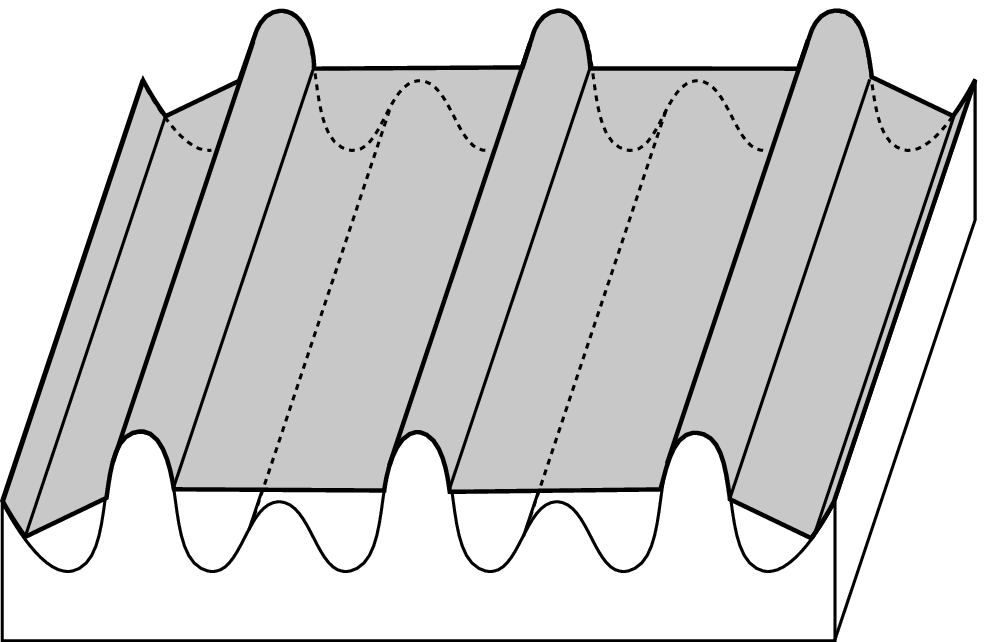}
			\end{center}
		\end{minipage}
	\end{tabular}
	\caption{minimizers in the case $\varepsilon=0$ \label{membrane}}
\end{figure}

The main goal of this paper is a formulation of a singular perturbation by the bending energy, that is, to characterize the limit of minimizers of $E_\varepsilon$ as $\varepsilon\downarrow0$ rigorously. 
To this end we utilize the notion of $\Gamma$-convergence established by De Giorgi \cite{DGFr75} in the 70's (more precisely, see \cite{Br06}, \cite{Br14}, \cite{DM93}, \cite{DM11}) which is a convergence of energy functionals for minimizing problems. 
The idea of this convergence is to identify the first nontrivial term in an asymptotic expansion for the energy of perturbed problems. 
Our main singular limit result is rigorously stated in Theorem \ref{gammathm}, but roughly speaking as follows:

\begin{theorem}[Singular limit]\label{mainthm1}
	If we define the singular limit energy $F$ by
	$$F[u]:=\int_{\partial\{u>\psi\}}4(\sqrt{2}-\sqrt{1+\alpha})\ d\mathcal{H}^0,$$
	then the $\Gamma$-convergence holds with respect to $W^{1,1}$-norm:
	$$\frac{1}{\varepsilon}(E_\varepsilon-\textstyle\inf_X{E_0})\xrightarrow{\ \Gamma\ } F\quad\text{as}\quad\varepsilon\downarrow0,$$
	where $X$ is a certain space of admissible functions $X\subset W^{1,1}(a,b)$.
\end{theorem}

\noindent
From this theorem, the following asymptotic expansion is also valid for our problem in some sense (see \cite{AnBa93}, \cite{BrTr08}):
$$(\inf{E_\varepsilon})=(\inf{E_0})+\varepsilon(\inf{F})+o(\varepsilon).$$
Thus it means that $F$ is the main effect of the bending energy and its quantity only depends on the $\mathcal{H}^0$-measure (the number) of the free boundary and the adhesion coefficient there. 
Since the adhesion coefficient at the free boundary determines the contact angle there ($\cos\theta=\alpha$) as mentioned above, it also means that the effect is determined by the number of ``edge'' singularities and their angles. 
Thanks to our singular limit result and geometrical intuition for the minimizing problem of $F$, we can easily find a key effect of the bending energy. 
For example, if $\alpha\equiv\text{const.}$ and there are multiple minimizers of $E_0$ then the one which has less number of edges shall be a minimizer of $E_\varepsilon$ when $\varepsilon>0$ is sufficiently small. 
More precisely, by Theorem \ref{mainthm1} and the fundamental theorem of $\Gamma$-convergence, we obtain the following:

\begin{corollary}\label{limitcoro}
	If $u^\varepsilon\to u$ in $W^{1,1}$ and any $u^\varepsilon$ is a minimizer of $E_\varepsilon$ in (\ref{problem1}), then $u$ minimizes $F$ among minimizers of $E_0$ in (\ref{problem1}).
\end{corollary}

\noindent It turns out that we can characterize the limit of minimizers of $(\ref{problem1})$ as $\varepsilon\downarrow0$ by our theorem.

An important point of our proof of Theorem \ref{mainthm1} is to prove the liminf condition of $\Gamma$-convergence (see Definition \ref{gammaconv}). 
The proof is mainly separated into two parts. 
The first part is to prove that it suffices to consider more regular sequences which are  ``close'' to a minimizer of $E_0$ in some sense. 
To state it rigorously, we introduce a notion called {\it $\delta$-associate} which explains closeness of functions. 
Especially, the part regarding how to coincide with $\psi$ is a key point because adhering or detaching leads to a discontinuous transition in our energy. 
We replace a general sequence by $\delta$-associate of $W^{2,1}$-regularity converging to a minimizer of $E_0$ so that all quantities in the energy are well-defined with no increase of the energy. 
The second part is to obtain a lower estimate for functions $\delta$-associated with a minimizer of $E_0$. 
By this procedure, we are able to handle the energy geometrically and establish a Modica-Mortola type inequality to prove the liminf condition.

\subsection{Related problems}

Now let us briefly survey some mathematical problems related to our problem from two viewpoints.

\subsubsection{Viewpoint of energies and settings}

In the non-adhesive case $\alpha\equiv1$, the energy in (\ref{problem1}) is consisting of the total squared curvature functional and the length functional:
\begin{eqnarray}\label{problem2}
	\varepsilon^2 \int_\gamma \kappa^2 ds + \int_\gamma ds,
\end{eqnarray}
so-called {\it Euler's elastic energy}. 
The critical points of the energy are called {\it elastica} (usually under boundary conditions and the length constraint for a curve $\gamma$). 
This problem is first considered by Euler in 1744 \cite{Euler1744}. 
Numerous authors have considered this variational problem or related problems under various constraints in order to analyze the configuration of elastic bodies (see \cite{Lo44}, \cite{Si08}, or \cite{Sa08} including a very well-written summary of the history of elastic problems by one section). 
However, there are still many unclear points in this problem because of difficulties of a higher order problem. 
Our problem is the variational problem for the Euler's elastic energy with the adhesion term under an obstacle-type constraint, thus it seems to have similar difficulties.
To circumvent such difficulties, we regard the bending energy as a perturbation in this paper.

Obstacle problems, which are variational problems under obstacle-type constraints, are motivating problems invoking free boundaries and have been studied for a long time (\cite{Ca98}, \cite{CaFr79}, \cite{PeShUr14}, \cite{Ro87}). 
A typical model is for the area functional (or the Dirichlet energy, the linearized version):
\begin{eqnarray}\label{problem3}
	\Minimize_{u\geq\psi}:  \int_\Omega\sqrt{1+|\nabla u|^2}\ dx\quad\left({\rm or}\quad\frac{1}{2}\int_\Omega|\nabla u|^2dx\right)
\end{eqnarray}
under a boundary condition.
Here $u$ is a function on a bounded smooth domain $\Omega\subset\mathbb{R}^n$. 
This is so-called the unilateral Plateaux problem. 
This problem invokes the free boundary $\partial{\{u>\psi\}}$ and in the non-coincidence set $\{u>\psi\}$ the graph of a solution $u$ is a minimal surface (or harmonic). 
Generally, if an energy is convex, bounded and coercive in a sense as (\ref{problem3}) then the classical variational inequality approach can work.
Thus we can obtain many fine properties of solutions (at least in the linearized case), for instance the uniquely existence, further the regularity of solution and of its free boundary \cite{Ca98}. 
The problem (\ref{problem3}) corresponds to the case no bending $\varepsilon=0$ and no adhesion $\alpha\equiv1$ in our problem (\ref{problem1}). 
In the problem (\ref{problem3}), we only obtain a natural singular limit even if the energy is perturbed by bending \cite[Theorem 9.5]{Ro87}.
Therefore, in the non-adhesive case, it seems that higher order terms can be neglected when $\varepsilon\ll1$ unlike our problem.

Our problem is closely related to the Alt-Caffarelli type variational problem \cite{AlCa81} which is a model of cavitation:
\begin{eqnarray}\label{problem4}
	\Minimize_{u}: \frac{1}{2}\int_\Omega|\nabla u|^2dx+\int_{\{u>0\}}Q^2dx
\end{eqnarray}
under a boundary condition, where $Q$ is a certain function. 
This is not an obstacle problem explicitly but the solutions are automatically constrained above zero, thus it is equivalent to the problem with the constraint $u\geq\psi\equiv0$ (flat obstacle problem). 
This problem is one of important interface models and has been generalized variously (for example, see \cite{MoWa14} generalizing the first term, \cite{ArTe13} the second term, or their references). 
Especially, Yamaura \cite{Ya94} considered a non-linearized case, namely the case that the first term is replaced to the area functional. 
Our problem is a generalization of the (one-dimensional) non-linearized Alt-Caffarelli problem regarding obstacle. 
Indeed, if $\varepsilon=0$ and $\psi\equiv0$ in (\ref{problem1}) then it is equivalent to
\begin{align*}
	\Minimize_{u\geq0}: \int_a^b\sqrt{1+u_x^2}\ dx+\int_{\{u>0\}}(1-\alpha)\ dx
\end{align*}
since $u_x\equiv0$ in $\{u=0\}$. 
Thus the result in this paper is particularly valid for (\ref{problem4}) with area functional instead of the Dirichlet energy and with continuous $Q:[a,b]\rightarrow(0,1)$.

\subsubsection{Viewpoint of singular perturbations}

In view of singular perturbation by $\Gamma$-convergence for variational problems, there are some related results especially in phase transition models.

One of the most celebrated results is, owing to Modica-Mortola \cite{MoMo77} (and \cite{Mo87}, \cite{St88}), for the energy arising from the van der Waals-Cahn-Hilliard theory of fluid-fluid phase transitions:
\begin{eqnarray}\label{problem5}
	\varepsilon\int_\Omega |\nabla u|^2dx+\frac{1}{\varepsilon}\int_\Omega W(u)dx,
\end{eqnarray}
where $u:\Omega\subset\mathbb{R}^n\rightarrow\mathbb{R}$ satisfying a volume constraint and $W$ is a double-well potential function, usually it is taken as $W(t):=(1-t^2)^2$ ($t\in\mathbb{R}$). 
They proved that the singular limit ($\Gamma$-limit) of the energy is proportional to the area of a transition layer. 
These works are generalized to the vector-valued case, anisotropic cases and multi-well potential cases (see \cite{BeBrRi05} and references cited there).

Furthernore, there are several higher order version results for the following energy:
\begin{eqnarray}\label{problem6}
	\varepsilon\int_\Omega |\nabla^2 u|^2dx+\frac{1}{\varepsilon}\int_\Omega W(\nabla u)dx.
\end{eqnarray}
One of the earliest studies is for the functional arising from the theory of smectic liquid crystals introduced by Aviles-Giga \cite{AvGi87}. 
The authors considered the case that $u:\Omega\subset\mathbb{R}^2\rightarrow\mathbb{R}$ satisfies certain boundary conditions and $W$ is a ``single-circle-well'' potential $W(\xi):=(1-|\xi|^2)^2$ ($\xi\in\mathbb{R}^2$). 

The energy (\ref{problem5}) or (\ref{problem6}) or similar one also arises from, for example the theory of solid-solid phase transitions, thin-films and magnetism.
There are several singular limit results for them (\cite{CoSc06}, \cite{GSMi13}, \cite{Ig12}, \cite{KoMu94}, \cite{SaSe07}). 
However, the number of results of higher order singular perturbations are limited compared with the first order cases.

Our result is one of the higher order singular perturbations and means that the problem (\ref{problem1}) can be regarded as a phase transition model as above.
The cost of transition between the phase $\{u=\psi\}$ and $\{u>\psi\}$ is determined by the place of a transition layer.

As mentioned above, a key point in our proof of the liminf condition of $\Gamma$-convergence is to reduce a general sequence to a sequence which is easy to handle. 
We mention that this concept resembles the ``slicing'' technique used in \cite{FoRy92}.

\subsection{Organization}

This paper is organized as follows: In \S\ref{presec} we prepare some notations and definitions. 
In \S\ref{ezerosec} we first consider the case $\varepsilon=0$ and derive some properties of minimizers of $E_0$. 
They are useful to prove our singular limit result. 
In \S\ref{gammaconvsec} we state our main singular limit theorem (Theorem \ref{gammathm}), and we prove it in \S\ref{liminfsec} and \S\ref{limsupsec}.

\section{Energy and function spaces} \label{presec}

In this section, we give an energy functional by introducing several notations. 
Let $\Omega\subset\mathbb{R}^n$ be a smooth bounded domain, $\psi:\overline{\Omega}\rightarrow\mathbb{R}$ be a smooth function and $g:\overline{\Omega}\rightarrow\mathbb{R}$ be a smooth function satisfying $g\geq\psi$ on a given $\mathcal{H}^{n-1}$-measurable subset $\Sigma$ of the boundary $\partial\Omega$, where $\mathcal{H}^{n-1}$ denotes the $(n-1)$-dimensional Hausdorff measure. 
For any positive integer $m$ and $1\leq p\leq\infty$, we define the space of admissible functions $X^{m,p}(\Omega)(= X^{m,p}_{\psi,g,\Sigma}(\Omega))$ by
\begin{eqnarray}\label{admset}
	X^{m,p}(\Omega) := \left\{u\in W^{m,p}(\Omega)\left|
	\begin{array}{c}
		u\geq\psi\text{ in }\Omega,\\
		u=g\text{ on }\Sigma
	\end{array}
	\right.\right\}.
\end{eqnarray}
The boundary condition is in the sense of trace. 
The space $X^{m,p}(\Omega)$ is a non-empty, convex and closed set in the Sobolev space $W^{m,p}(\Omega)$. 
Usually it is assumed that the partial boundary $\Sigma$ is not empty, however we do not assume in this paper since our problem is not trivial even if $\Sigma=\emptyset$.

Let $\alpha:[a,b]\rightarrow\mathbb{R}$ be a continuous function satisfying $0<\underline{\alpha}\leq\overline{\alpha}<1$, where $\underline{\alpha}:=\min\alpha$ and $\overline{\alpha}:=\max\alpha$.
For $\varepsilon\geq0$, we define the energy functional $E_\varepsilon$ by
\begin{eqnarray}\label{eeps}
	E_\varepsilon[u]:=\varepsilon^2\int_\Omega H_u^2\sqrt{1+|\nabla u|^2}\ dx + \int_\Omega \tilde{\alpha}[u] \sqrt{1+|\nabla u|^2}\ dx,
\end{eqnarray}
where $H_u$ is the mean curvature of $u$:
$$H_u:=\text{div}{\left(\cfrac{\nabla u}{\sqrt{1+|\nabla u|^2}}\right)},$$
the coefficient $\tilde{\alpha}$ is the redefined adhesion coefficient:
$$\tilde{\alpha}[u](\cdot):=\vartheta_\alpha(\ \cdot\ ,u(\cdot)-\psi(\cdot))$$
and $\vartheta_\alpha:[a,b]\times\mathbb{R}\rightarrow[\underline{\alpha},1]$ such that
\begin{align*}
	\vartheta_\alpha(x,y) := 
	\begin{cases}
		1 & (y>0),\\
		\alpha(x) & (y\leq 0).
	\end{cases}
\end{align*}
$E_\varepsilon$ is well-defined on $W^{2,1}(\Omega)$ for any $\varepsilon\geq0$ and especially $E_0$ is well-defined on $W^{1,1}(\Omega)$. Throughout this paper, we fix $\Omega$, $\psi$, $g$, $\Sigma$ and $\alpha$.

\begin{definition}
	Let $u\in X^{m,p}(\Omega)\cup C(\overline{\Omega})$. We set
	\begin{align*}
		\Omega_0^u&:=\{u=\psi\}=\{x\in\Omega \mid u(x)=\psi(x)\}\subset\Omega,\\
		\Omega_+^u&:=\{u>\psi\}=\{x\in\Omega \mid u(x)>\psi(x)\}\subset\Omega,\\
		\partial\Omega_+^u&:=\partial\{u>\psi\}\cap\Omega\subset\Omega.
	\end{align*}
	We call $\Omega_0^u$ coincidence set, $\Omega_+^u$ non-coincidence set and $\partial\Omega_+^u$ free boundary.
	Note that the non-coincidence set $\Omega_+^u$ is open in $\Omega$ and the coincidence set $\Omega_0^u$ and the free boundary $\partial\Omega_+^u$ are closed in $\Omega$.
	Moreover $\Omega=\Omega_0^u\sqcup\Omega_+^u$ holds.
\end{definition}

\section{Minimizers of $E_0$}\label{ezerosec}

In this section, we derive some properties of minimizers of $E_0$ in the one-dimensional case. 
In \S\ref{lscsec} we verify the lower semicontinuity of $E_0$ and in \S\ref{1-dsec} we derive properties about the shape of minimizers of $E_0$. 
These are useful to prove Theorem \ref{gammathm} which is our main singular limit result.

\subsection{Lower semicontinuity of $E_0$}\label{lscsec}

\begin{lemma}\label{lscgeneral}
	Let $\psi:\overline{\Omega}\rightarrow\mathbb{R}$ be a smooth function, $h:\Omega\times\mathbb{R}\times\mathbb{R}^n\rightarrow[0,\infty)$ be a Borel function and $E[u]:=\int_{\Omega}h(x,u-\psi,\nabla u)dx$. 
	If $h(x,\cdot,\cdot)$ is lower semicontinuous for $a.e.$ $x\in\Omega$, then $E:W^{1,1}(\Omega)\rightarrow[0,\infty]$ is lower semicontinuous.
\end{lemma}

\begin{proof}
	Fix any convergent sequence $u^{\varepsilon}\rightarrow u$ in $W^{1,1}$. 
	For any subsequence, there exist a subsequence such that
	$$u^{{\varepsilon}_j}\rightarrow u\quad\text{and}\quad \nabla u^{{\varepsilon}_j}\rightarrow \nabla u\quad a.e.\text{ in }\Omega.$$ 
	Then
	\begin{align*}
		\liminf_{j \to \infty}E[u^{{\varepsilon}_j}] &= \liminf_{j \to \infty}\int_\Omega h(x,u^{{\varepsilon}_j}-\psi,\nabla u^{{\varepsilon}_j})\ dx\\
		&\geq \int_\Omega\liminf_{j \to \infty} h(x,u^{{\varepsilon}_j}-\psi,\nabla u^{{\varepsilon}_j})\ dx\\
		&\geq \int_\Omega h(x,u-\psi,\nabla u)\ dx = E[u].
	\end{align*}
	The first inequality follows by Fatou's lemma, and the second one follows by the lower semicontinuity of $h$.
	Thus we get the consequence.
\end{proof}

\begin{proposition}\label{lscezero}
	$E_0: X^{1,1}(\Omega)\rightarrow[0,\infty)$ is lower semicontinuous.
\end{proposition}

\begin{proof}
	By taking $h(x,y,\xi):=\vartheta_\alpha(x,y)\sqrt{1+|\xi|^2}$ in Lemma \ref{lscgeneral} and the lower semicontinuity of $\vartheta_\alpha(x,\cdot)$.
\end{proof}

\subsection{Properties in the one-dimensional case}\label{1-dsec}
Now we assume $n=1$ and $\Omega=(a,b)$, bounded open interval.
Note that $X^{m,p}(a,b)$ is contained in $C^{m-1}([a,b])$ in this case. 
To state the main theorem in this subsection, we introduce some terminologies.

\begin{definition}[Partitional regularity and contact angle]\quad
	\begin{enumerate}
		\item Let $N$ be a nonnegative integer. We say that $\{x_i\}_{i=0}^{N+1}\subset\mathbb{R}$ is a {\it partition} of an interval $(a,b)$ in $\mathbb{R}$ if $a=x_0<x_1<\cdots<x_{N+1}=b$.
		\item Let $u\in X^{1,1}(a,b)$.
		We say that $u$ is {\it partitional regular} if there exist a nonnegative integer $N$ and a partition $\{x_i\}_{i=0}^{N+1}$ of $(a,b)$ such that:
		\begin{enumerate}
			\item the interval $(x_0,x_1)$ is contained in either $(a,b)_0^{u}$ or $(a,b)_+^{u}$,
			\item if $(x_{i-1},x_i)\subset(a,b)_0^u$ then $(x_i,x_{i+1})\subset(a,b)_+^u$,
			\item if $(x_{i-1},x_i)\subset(a,b)_+^u$ then $(x_i,x_{i+1})\subset(a,b)_0^u$.
		\end{enumerate}
		Note that $\partial(a,b)_+^u=\{x_1,\ldots,x_N\}$ and $u\equiv\psi$ or $u>\psi$ in $(x_i,x_{i+1})$ alternately.
		We denote such regularity by {\it $[x_0;\ldots;x_{N+1}]$-regular} when we want to show the partition explicitly.
		\item Let $u$ be a function.
		If the derivative $u_x$ exists at $x$, we define the {\it tangent angle} of $u$ at $x$ by $\theta_u(x):=\arctan u_{x}(x)\in(-\pi/2,\pi/2)$.
		\item Let $u\in X^{1,1}(a,b)$ be $[x_0;\ldots;x_{N+1}]$-regular and $N$ be positive. 
		For $1\leq i\leq N$ with $(x_i,x_{i+1})\subset(a,b)_+^u$ (resp. $(a,b)_0^u$) we define the {\it contact angle $\theta_i\in[0,\pi)$} of $u$ with $\psi$ at $x_i$ by
		$$\theta_i:=\theta_u(x_i+)-\theta_\psi(x_i)\quad (\text{resp. }\theta_\psi(x_i)-\theta_u(x_i-))$$
		if $\theta_u(x_i\pm):=\lim_{x\to x_i\pm0}\theta_u(x)$ can be defined.
	\end{enumerate}
\end{definition}

Denote the set of minimizers of a functional $E:X^{m,p}(\Omega)\rightarrow[0,\infty]$ by $$\text{argmin}_{X^{m,p}}E:=\left\{u\in X^{m,p}(\Omega)\left|E[u]=\inf_{X^{m,p}(\Omega)}E\right\}\right..$$
The following is the main theorem in this subsection. 

\begin{theorem} \label{regezero}
	Let $\Omega=(a,b)$.
	If $\bar{u}\in{\rm argmin}_{X^{1,1}}E_0$ then there exists a partition $\{\bar{x}_i\}_{i=0}^{N+1}$ of $(a,b)$ such that $\bar{u}$ is $[\bar{x}_0;\ldots;\bar{x}_{N+1}]$-regular, the graph of $\bar{u}$ is a segment on any interval $(\bar{x}_i,\bar{x}_{i+1})\subset(a,b)_+^{\bar{u}}$ and the contact angle $\theta_i$ of $\bar{u}$ with $\psi$ at $\bar{x}_i$ satisfies $\cos\theta_i=\alpha(\bar{x}_i)$ for any $1\leq i\leq N$.
\end{theorem}

\begin{remark}\label{regrem}
	This theorem especially means that the regularity of any minimizer of $E_0$ in $X^{1,1}(a,b)$ is up to $X^{1,\infty}(a,b)$ (further, piecewise smooth) and this is optimal.
	Moreover, it is remarkable that the number of the connected components of the non-coincidence set of any minimizer is finite.
	This is not valid without the adhesion energy.
\end{remark}

To prove Theorem \ref{regezero}, we prepare several lemmas.

\begin{lemma}\label{psidelta}
	There exists $\delta_{\psi,\alpha}>0$ depending on $\psi$ and $\alpha$ such that for any interval $I\subset(a,b)$ whose width is not larger than $\delta_{\psi,\alpha}$ the inequality holds:
	$$\inf_{I}\sqrt{1+{\psi_x}^2}-\overline{\alpha}\sup_{I}\sqrt{1+{\psi_x}^2}\geq\frac{1-\overline{\alpha}}{2}.$$
\end{lemma}

\begin{proof}
	For any interval $I$, we have
	\begin{align*}
		& \inf_{I}\sqrt{1+{\psi_x}^2}-\overline{\alpha}\sup_{I}\sqrt{1+{\psi_x}^2}\\
		&\geq \overline{\alpha}\left(\frac{1-\overline{\alpha}}{\overline{\alpha}}-\left(\sup_{I}\sqrt{1+{\psi_x}^2}-\inf_{I}\sqrt{1+{\psi_x}^2}\right)\right).
	\end{align*}
	Since $\psi_x$ is uniformly continuous, if the width of $I$ is sufficiently small then
	$$\sup_{I}\sqrt{1+{\psi_x}^2}-\inf_{I}\sqrt{1+{\psi_x}^2}\leq\frac{1-\overline{\alpha}}{2\overline{\alpha}},$$
	thus we get the consequence.
\end{proof}

\begin{lemma}\label{reglem1}
	Let $u\in X^{1,1}(a,b)$, $0<\delta\leq\delta_{\psi,\alpha}$ and $I\subset(a,b)$ be an interval whose width is $\delta$. Suppose that $u=\psi$ on $\partial I$ and $u>\psi$ at some interior point of $I$. Then the function
	\begin{align*}v:=
		\begin{cases}
			\psi & \text{in }I,\\
			u & \text{in }(a,b) \setminus I,
		\end{cases}
	\end{align*}
	belongs to $X^{1,1}(a,b)$ and satisfies $E_0[u]>E_0[v]$.
\end{lemma}

\begin{proof}
	The function $v$ obviously belongs to $X^{1,1}(a,b)$ since $u=\psi$ on $\partial I$, thus we prove the inequality $E_0[u]>E_0[v]$.\par
	We first prove the case that $u>\psi$ at any interior point of $I$.
	Denote the endpoints of $I$ by $y_0$ and $y_1$.
	Then we have
	\begin{align*}
		&\ E_0[u]-E_0[v] \\
		=&\ \int_{y_0}^{y_1}\sqrt{1+{u_x}^2}\ dx-\int_{y_0}^{y_1}\alpha\sqrt{1+{\psi_x}^2}\ dx\\
		\geq&\ \sqrt{|y_1-y_0|^2+|u(y_1)-u(y_0)|^2}-\overline{\alpha}|y_1-y_0|\max_{[y_0,y_1]}\sqrt{1+{\psi_x}^2}\\
		=&\ \sqrt{|y_1-y_0|^2+|\psi(y_1)-\psi(y_0)|^2}-\overline{\alpha}|y_1-y_0|\max_{[y_0,y_1]}\sqrt{1+{\psi_x}^2}\\
		\geq&\ |y_1-y_0|\left(\min_{[y_0,y_1]}\sqrt{1+{\psi_x}^2}-\overline{\alpha}\max_{[y_0,y_1]}\sqrt{1+{\psi_x}^2}\right)\\
		\geq&\ \delta\cdot\frac{1-\overline{\alpha}}{2}.
	\end{align*}
	The last inequality follows by Lemma \ref{psidelta} and this implies the consequence.\par
	Next we prove the general case.
	Since $u$ is continuous, we can decompose $I\cap(a,b)_+^u(\neq\emptyset)$ into at most countable disjoint intervals $I_i$, where $1\leq i\leq N$ and $N$ is a positive integer or $\infty$.
	Note that $u>\psi$ in any $I_i$ and $u=\psi$ at any endpoint of $I_i$.
	Denote $v_0:=u$ and for all $1\leq i\leq N$ define $v_i\in X^{1,1}(a,b)$ by
	\begin{align*}
		v_i:=
		\begin{cases}
			\psi & \text{in }I_i,\\
			v_{i-1} & \text{in }(a,b)\setminus I_i.
		\end{cases}
	\end{align*}
	By the above paragraph, we have $E_0[v_{i-1}]>E_0[v_i]$ for any $1\leq i\leq N$.
	Therefore, if $N$ is finite then we have $v=v_N$ and $E_0[u]>E_0[v]$.
	If $N=\infty$ then we have $v_i\rightarrow v$ in $X^{1,1}(a,b)$ and, by Proposition \ref{lscezero},
	$$E_0[u]>\liminf_{i\to\infty}E[v_i]\geq E_0[v].$$
	The proof is completed.
\end{proof}

\begin{lemma}\label{minilem}
	If $\bar{u}\in{\rm argmin}_{X^{1,1}}E_0$ then $\bar{u}_{xx}\equiv0$ in $(a,b)_{+}^{\bar{u}}$, that is, the graph of $\bar{u}$ is a segment on any connected component of $(a,b)_+^{\bar{u}}$.
\end{lemma}

\begin{proof}
	Since $\bar{u}$ is continuous, the set $(a,b)_{+}^{\bar{u}}$ is open, thus the graph of $\bar{u}$ is a minimal surface in $(a,b)_{+}^{\bar{u}}$.
\end{proof}

We now prove Theorem \ref{regezero}.

\begin{proof}[Theorem \ref{regezero}]
	Fix a minimizer $\bar{u}$.
	We first prove by contradiction that the numbers of the connected components of $(a,b)_+^{\bar{u}}$ and $(a,b)_0^{\bar{u}}$ are finite.
	If either of them is infinite, then so is the other, thus there exists a sufficiently small interval $I$ as in Lemma \ref{reglem1}.
	This contradicts the minimality of $\bar{u}$.\par
	Next we prove any connected component of $(a,b)_0^{\bar{u}}$ is not a point (but an interval) by contradiction.
	If there is a connected component $\{x^*\}\subset(a,b)_0^{\bar{u}}$ then there exists $r^*>0$ such that the intervals $[x^*-r^*,x^*)$ and $(x^*,x^*+r^*]$ are contained in $(a,b)_+^{\bar{u}}$.
	Note that $\bar{u}$ is straight on $[x^*-r^*,x^*)$ and $(x^*,x^*+r^*]$ by Lemma \ref{minilem}.
	Moreover, since $u\geq\psi$, the following inequalities hold:
	\begin{align}\label{regezerocond}
		\bar{u}_{x+}(x^*)\geq\psi_x(x^*)\geq\bar{u}_{x-}(x^*).
	\end{align}
	If $\bar{u}_{x+}(x^*)>\bar{u}_{x-}(x^*)$ holds then the function $v^{r^*}$ which is equal to $\bar{u}$ out of $[x^*-r^*,x^*+r^*]$ and a segment in $[x^*-r^*,x^*+r^*]$ belongs to $X^{1,1}(a,b)$ and $E_0[\bar{u}]>E_0[v^{r^*}]$ holds.
	This contradicts the minimality of $\bar{u}$.
	Thus we may assume that the equalities are attained in (\ref{regezerocond}).
	For $0<r<r^*$ we define
	\begin{align*}
		v^r:=
		\begin{cases}
			\bar{u} & \text{in }(a,x^*]\cup[x^*+r^*,b),\\
			\psi & \text{in } (x^*,x^*+r],\\
			\text{segment} & \text{in }(x^*+r,x^*+r^*).
		\end{cases}
	\end{align*}
	Then there exists a sequence $r'\downarrow0$ such that $v^{r'}\in X^{1,1}(a,b)$ and we see
	\begin{align*}
		E_0[v^{r'}]-E_0[\bar{u}]&=\int_{x^*}^{x^*+r'}\alpha\sqrt{1+\psi_x^2}\ dx\\
		&+\sqrt{|r^*-r'|^2+|\bar{u}(x^*+r^*)-\psi(x^*+r')|^2}\\
		&-\sqrt{|r^*|^2+|\bar{u}(x^*+r^*)-\bar{u}(x^*)|^2}.
	\end{align*}
	Noting that $\bar{u}(x^*)=\psi(x^*)$ by $x^*\in(a,b)_0^{\bar{u}}$, the equality assumption of (\ref{regezerocond}) implies $\bar{u}_{x+}(x^*)=\psi_x(x^*)$ and the relation $r^*\bar{u}_{x+}(x^*)=\bar{u}(x^*+r^*)-\bar{u}(x^*)$ follows since $\bar{u}$ is straight on $(x^*,x^*+r^*]$, we have
	$$\lim_{r'\downarrow0}\frac{E_0[v^{r'}]-E_0[\bar{u}]}{r'}=\alpha(x^*)\sqrt{1+\bar{u}_{x+}^2(x^*)}-\sqrt{1+\bar{u}_{x+}^2(x^*)}.$$
	By $\alpha(x^*)<1$, the above is negative.
	This also contradicts the minimality of $\bar{u}$.
	Therefore, we find that any connected component of $(a,b)_0^{\bar{u}}$ is an interval.\par
	Finally, we compute the contact angles of $\bar{u}$.
	From the above arguments, we know that $\bar{u}$ is $[\bar{x}_0,\dots,\bar{x}_{N+1}]$-regular for some partition.
	Assume $N>0$ and fix any $1\leq i\leq N$.
	We only consider the case $(\bar{x}_i,\bar{x}_{i+1})\subset(a,b)_+^{\bar{u}}$ (the case $(\bar{x}_i,\bar{x}_{i+1})\subset(a,b)_0^{\bar{u}}$ is similar).
	Since the partition is finite, there exists $r_i^*>0$ such that $[\bar{x}_i-r_i^*,\bar{x}_i]\subset(a,b)_0^{\bar{u}}$ and $(\bar{x}_i,\bar{x}_i+r_i^*]\subset(a,b)_+^{\bar{u}}$.
	Note that $\bar{u}$ is straight on $(\bar{x}_i,\bar{x}_i+r_i^*]$ and $\psi_x(\bar{x}_i)=\bar{u}_{x-}(\bar{x}_i)$.
	Notice that $\bar{u}_{x+}(\bar{x}_i)>\psi_x(\bar{x}_i)$ by considering as above paragraph.
	Thus we can perturb $\bar{u}$ near $\bar{x}_i$ for small $\pm r\in(0,r_i^*)$ as above.
	Denoting the perturbed $\bar{u}$ by $v^r$, we can similarly compute as
	$$\lim_{r\to\pm0}\frac{E_0[v^r]-E_0[\bar{u}]}{r}=\alpha(\bar{x}_i)\sqrt{1+\psi_x^2(\bar{x}_i)}-\frac{1+\psi_x(\bar{x}_i)\bar{u}_{x+}(\bar{x}_i)}{\sqrt{1+\bar{u}^2_{x+}(\bar{x}_i)}}.$$
	Since $\bar{u}$ is a minimizer, the above quantity has to be zero, thus we obtain
	\begin{align*}
		\alpha(\bar{x}_i)&=\frac{1+\psi_x(\bar{x}_i)\bar{u}_{x+}(\bar{x}_i)}{\sqrt{1+\psi_x^2(\bar{x}_i)}\sqrt{1+\bar{u}^2_{x+}(\bar{x}_i)}}\\
		&= \cos{\theta_u(\bar{x}_i+)}\cos{\theta_\psi(\bar{x}_i)}+\sin{\theta_u(\bar{x}_i+)}\sin{\theta_\psi(\bar{x}_i)}\\
		&= \cos{(\theta_u(\bar{x}_i+)-\theta_\psi(\bar{x}_i))}=\cos{\theta_i}.
	\end{align*}
	This completes the proof.
\end{proof}

\begin{remark}\label{bdryangle}
	If $\bar{u}(a)=\psi(a)$ and $(x_0,x_1)\subset(a,b)_+^{\bar{u}}$ ($x_0=a$), then we obtain $\theta_{\bar{u}}(a+)-\theta_{\psi}(a+)\geq\alpha(a)$ by a similar consideration.
	We also obtain $\theta_{\psi}(b-)-\theta_{\bar{u}}(b-)\geq\alpha(b)$.
\end{remark}

\begin{remark}\label{locminrem}
	Since we only used local perturbations in the above arguments, the results in this section are also valid for local minimizers $\bar{u}$ of $E_0$, that is, there exists $\varepsilon>0$ such that $E_0[u]\geq E_0[\bar{u}]$ for any $\|u-\bar{u}\|_{W^{1,1}}<\varepsilon$.
\end{remark}

\section{$\Gamma$-convergence}\label{gammaconvsec}

Now we rigorously state our main $\Gamma$-convergence result in the one-dimensional setting.
We set $\Omega=(a,b)$ and define $F_\varepsilon,F:X^{1,1}(a,b)\rightarrow[0,\infty]$ by
\begin{eqnarray}\label{feps1}
	&&F_\varepsilon[u] :=
	\begin{cases}
		\begin{array}{l}
			\varepsilon\displaystyle\int_{a}^{b} \kappa_u^2\sqrt{1+{u_x}^2}\ dx \\
			\qquad+\cfrac{1}{\varepsilon}\left(E_0[u]-\displaystyle\inf_{X^{1,1}}E_0\right)
		\end{array}
		& (u\in X^{2,1}(a,b)),\\
		\infty & (\text{otherwise}),
	\end{cases}
\end{eqnarray}
where $\kappa_u=\cfrac{u_{xx}}{(1+u_x^2)^{3/2}}$ is curvature, and
\begin{eqnarray}\label{fzero1}
	&&F[u]:=\begin{cases}
		\displaystyle\int_{\partial(a,b)_{+}^u} 4(\sqrt{2}-\sqrt{1+\alpha})\ d\mathcal{H}^0 & (u\in{\rm argmin}_{X^{1,1}}E_0),\\
		\infty & (\text{otherwise}).
	\end{cases}
\end{eqnarray}
We begin by recalling the definition of $\Gamma$-convergence.

\begin{definition}[$\Gamma$-convergence]
	\label{gammaconv}
	Let $X$ be a metric space and $F_\varepsilon,F:X\rightarrow[0,\infty]$.
	We say that $F_\varepsilon$ {\it $\Gamma(X)$-converges} to $F$ as $\varepsilon \downarrow 0$ if the following conditions hold:
	\begin{enumerate}
		\item For any convergent sequence $u_\varepsilon \rightarrow u$ ($\varepsilon \downarrow 0$) in $X$, 
		$$\liminf_{\varepsilon \to 0}F_\varepsilon[u_\varepsilon]\geq F[u].$$
		\item For any $u \in X$, there exists a convergent sequence $u_\varepsilon\rightarrow u$ ($\varepsilon \downarrow 0$) in $X$ such that 
		$$\limsup_{\varepsilon \to 0}F_\varepsilon[u_\varepsilon]\leq F[u].$$
	\end{enumerate} 
	We denote such convergence by $F_{\varepsilon}\xrightarrow{\Gamma}F$ on $X$.
\end{definition}

We are now in position to state our main result on $\Gamma$-convergence.

\begin{theorem} \label{gammathm}
	Let $F_\varepsilon$ and $F$ be the functionals on $X^{1,1}(a,b)$ defined by {\rm (\ref{feps1})} and {\rm (\ref{fzero1})}.
	Then $F_{\varepsilon}\xrightarrow{\Gamma}F$ on $X^{1,1}(a,b)$.
\end{theorem}

This theorem is proved in \S\ref{liminfsec} and \S\ref{limsupsec}.

It is one of the important properties of $\Gamma$-convergence that if $F_\varepsilon\xrightarrow{\Gamma}F$ and $u_\varepsilon$ is a minimizer of $F_\varepsilon$ then any cluster point of $\{u_\varepsilon\}$ is a minimizer of $F$.
In our setting, the minimizing problem of $F_\varepsilon$ and of $E_\varepsilon$ are equivalent, thus any cluster point of $\{\bar{u}_\varepsilon\}$ which is a sequence of minimizers of $E_\varepsilon$ is a minimizer of $F$ (as mentioned in Corollary \ref{limitcoro}).

\section{Lim-inf condition of $\Gamma$-convergence} \label{liminfsec}

In this section, we prove the liminf condition of Theorem \ref{gammathm}, which is the first condition of Definition \ref{gammaconv}.
The proof is separated into three parts: \S\ref{liminfsub1}--\ref{liminfsub3}.
In \S\ref{liminfsub1} we show that we only have to check special sequences.
In \S\ref{liminfsub2} we impose an essential and stronger restriction for sequences in order to obtain a lower estimate.
By this restriction, we are able to consider the energy $F_\varepsilon$ geometrically.
In \S\ref{liminfsub3} we obtain a lower bound for such restricted sequences.
A summary of the overall proof is given in \S\ref{liminfsub4}.

\subsection{Assumption on sequences}\label{liminfsub1}
We first give some simple assumptions for sequences.

\begin{assumption}\label{assumpseq} 
	Assume that a sequence
	$u^{\varepsilon}\rightarrow u$ in $X^{1,1}(a,b)$ satisfies
	\begin{itemize}\setlength{\leftskip}{3pt}
		\item[(1)] $\liminf_{{\varepsilon}\to 0}F_{\varepsilon}[u^{\varepsilon}]<\infty$,
		\item[(2)] $F_{\varepsilon}[u^{\varepsilon}]<\infty$ for all $\varepsilon>0$ (especially $\{u^{\varepsilon}\}\subset X^{2,1}(a,b)$),
		\item[(3)] $u=\bar{u}\in{\rm argmin}_{X^{1,1}}E_0$ (especially ${\rm argmin}_{X^{1,1}}E_0\neq\emptyset$),
		\item[(4)] $\bar{u}$ is $[\bar{x}_0;\ldots;\bar{x}_{N+1}]$-regular for some $N>0$.
	\end{itemize}
\end{assumption} 

\begin{proposition}\label{restseq}
	If the liminf condition of Theorem \ref{gammathm} is valid for sequences $u^{\varepsilon}\rightarrow u$ satisfying Assumption \ref{assumpseq}, then the liminf condition is fulfilled for any sequence.
\end{proposition}
\begin{proof}
	(1) If $\liminf_{{\varepsilon}\to 0}F_{\varepsilon}[u^{\varepsilon}]=\infty$ then the liminf condition is trivial.\\
	(2) If there exists $\varepsilon_0>0$ such that $F_{\varepsilon}[u^{\varepsilon}]=\infty$ for any $0<\varepsilon<\varepsilon_0$, then $\liminf_{{\varepsilon}\to 0}F_{\varepsilon}[u^{\varepsilon}]=\infty$, thus it is trivial.
	If not, then there exists a subsequence $\{u^{\varepsilon'}\}$ such that $F_{\varepsilon'}[u^{\varepsilon'}]<\infty$ for all $\varepsilon'$ and
	$$\liminf_{{\varepsilon'}\to 0}F_{\varepsilon'}[u^{\varepsilon'}]\leq\liminf_{{\varepsilon}\to 0}F_{\varepsilon}[u^{\varepsilon}],$$
	thus we can assume (2) without loss of generality.\\
	(3) If $u^{\varepsilon}\rightarrow u$ and $E_0[u]>\inf_{X^{1,1}}E_0$, then $\liminf_{{\varepsilon}\to 0}F_{\varepsilon}[u^{\varepsilon}]=\infty$ by Proposition \ref{lscezero}, thus it is trivial.\\
	(4) If $\bar{u}\in X^{1,1}(a,b)$ is $[x_0;\ldots;x_{N+1}]$-regular then
	\begin{align}\label{repsinglim}
		\int_{\partial(a,b)_{+}^u} 4(\sqrt{2}-\sqrt{1+\alpha})\ d\mathcal{H}^0 = \sum_{i=1}^N 4\left(\sqrt{2}-\sqrt{1+\alpha(x_i)}\right),
	\end{align}
	thus it is trivial in the case $N=0$.
	Here the sum is defined to be zero when $N=0$.
\end{proof}

\subsection{$\delta$-associate}\label{liminfsub2}

In this subsection, we reduce sequences in the liminf condition of Theorem \ref{gammathm} while not increasing the limit inferior.
This reduction is important because it becomes easy to obtain a lower estimate for such reduced sequences.
To state the reduction result, we introduce a notation.

\begin{definition}[$\delta$-associate]\label{deltaasso}
	Let $\delta>0$ and $u\in X^{1,1}(a,b)$.
	We say that $v\in X^{1,1}(a,b)$ is {\it $\delta$-associated} with $u$ if $\|u-v\|_\infty\leq\delta$ and there exists a continuous strictly increasing operator $S:[a,b]\rightarrow[a,b]$ such that $|Sx-x|\leq\delta$ for all $x\in[a,b]$ and $S[(a,b)_+^u]=(a,b)_+^v$.
\end{definition}

\begin{figure}[htbp]
	\begin{center}
		\def\svgwidth{60mm}
		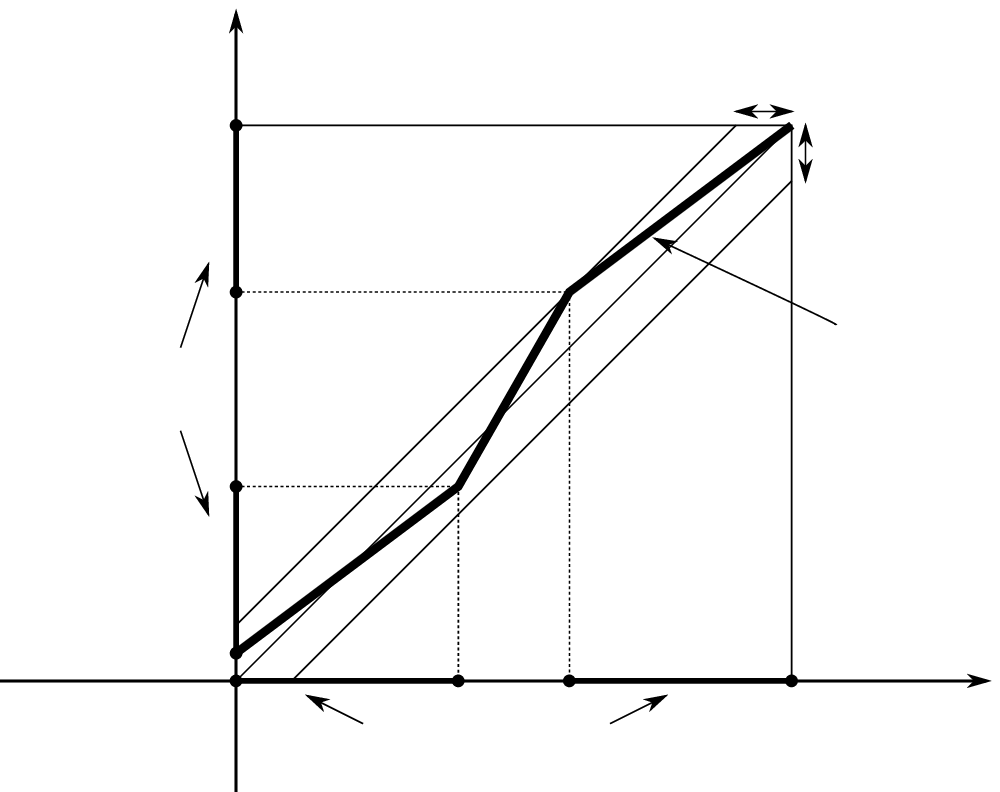
		\label{deltaS}
	\end{center}
	\begin{tabular}{cc}
		\begin{minipage}{0.5\hsize}
			\begin{center}
				\def\svgwidth{50mm}
				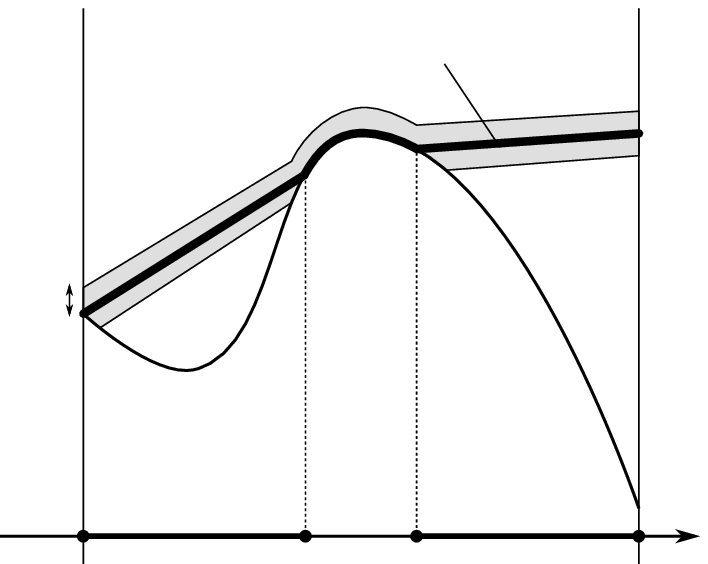
			\end{center}
		\end{minipage}
		\begin{minipage}{0.5\hsize}
			\begin{center}
				\def\svgwidth{50mm}
				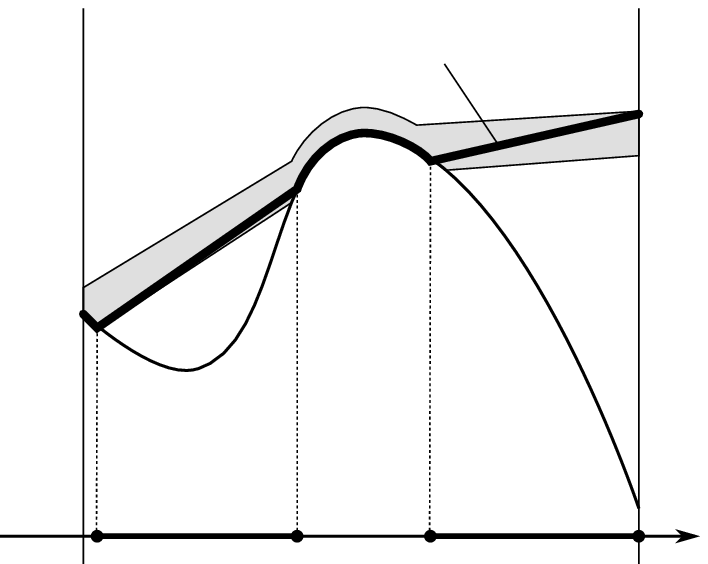
			\end{center}
		\end{minipage}
	\end{tabular}
	\caption{$\delta$-associate}
	\label{deltauv}
\end{figure}

\begin{remark}
	For any partitional regular function $u$ and function $v$ which is $\delta$-associated with $u$, the numbers of the connected components of $(a,b)_+^u$ and $(a,b)_+^v$ are equal and $v$ is partitional regular.
\end{remark}

The following proposition is the main reduction result of this subsection.

\begin{proposition}\label{lbprop1}
	Let $u^{\varepsilon}\rightarrow\bar{u}$ in $X^{1,1}(a,b)$ as in Assumption \ref{assumpseq}.
	Then there exists $\bar{\delta}:=\bar{\delta}(\psi,\alpha,\bar{u})>0$ satisfying the following: for any $0<\delta\leq\bar{\delta}$ there exists a sequence $\{w^{{\varepsilon}_j}_\delta\}_j\subset X^{2,1}(a,b)$ such that,
	\begin{enumerate}
		\item $\liminf_{j\to\infty}F_{{\varepsilon}_j}[w^{{\varepsilon}_j}_\delta]\leq\liminf_{{\varepsilon}\to 0}F_{\varepsilon}[u^{\varepsilon}]$,
		\item for any $j$ the function $w^{{\varepsilon}_j}_\delta$ is $\delta$-associated with $\bar{u}$.
	\end{enumerate}
\end{proposition}

\begin{proof}
	This is a direct consequence of the following three lemmas: Lemma \ref{redlem1}, Lemma \ref{redlem2} and Lemma \ref{redlem3}.
\end{proof}

\begin{lemma}\label{redlem1}
	Let $u^{\varepsilon}\rightarrow\bar{u}$ in $X^{1,1}(a,b)$ as in Assumption \ref{assumpseq} and $\delta_{\psi,\alpha}>0$ in Lemma \ref{psidelta}.
	Then for any $0<\delta\leq3\delta_{\psi,\alpha}$ there exists a subsequence $\{u_\delta^{{\varepsilon}_j}\}_j\subset\{u^{\varepsilon}\}_\varepsilon$ such that,
	\begin{enumerate}
		\item $\liminf_{j\to\infty}F_{{\varepsilon}_j}[u_\delta^{{\varepsilon}_j}]=\liminf_{{\varepsilon}\to 0}F_{\varepsilon}[u^{\varepsilon}]$,
		\item for any ${\varepsilon}_j>0$ and interval $I\subset(a,b)_0^{\bar{u}}$ with width $\delta$ there exists $y\in I$ such that $u_\delta^{{\varepsilon}_j}(y)=\psi(y)$.
	\end{enumerate}
\end{lemma}

\begin{proof}
	Fix any $0<\delta\leq3\delta_{\psi,\alpha}$.
	By definition of $\liminf$, we can take a subsequence $\{u^{\varepsilon'}\}\subset\{u^{\varepsilon}\}$ such that 
	$$\lim_{\varepsilon'\to0}F_{{\varepsilon}'}[u^{{\varepsilon}'}]=\liminf_{{\varepsilon}\to 0}F_{\varepsilon}[u^{\varepsilon}]<\infty.$$
	We prove that this sequence has a subsequence satisfying the second condition of Lemma \ref{redlem1} by contradiction.
	Note that it suffices to prove that, for any $I\subset(a,b)_0^{\bar{u}}$ whose width is $\delta/3$, there exists a subsequence (depending on $I$) touching $\psi$ somewhere in $I$.
	Indeed, since $(a,b)_0^{\bar{u}}$ is covered by a finite number of intervals in $(a,b)_0^{\bar{u}}$ with width $\delta/3$, if we proceed to take such subsequences repeatedly for each interval of the covering, then any function of the eventual subsequence touches $\psi$ in each interval of the covering.
	The consequence follows since any interval in $(a,b)_0^{\bar{u}}$ with width $\delta$ contains some interval of the covering. \par
	Thus we suppose that for a given $I$ there is no such subsequence.
	Namely, suppose that there would exist $\varepsilon'>0$ and some interval $I\subset(a,b)_0^{\bar{u}}$ with width $\delta/3$ such that $u^{\varepsilon''}>\psi$ in $I$ for any $0<\varepsilon''<\varepsilon'$. 
	Then we would obtain
	\begin{align*}
		\varepsilon''F_{\varepsilon''}[u^{\varepsilon''}]&\geq E_0[u^{\varepsilon''}]-E_0[\bar{u}]\\
		&\geq \int_{I} \sqrt{1+(u_{x}^{\varepsilon''})^2}\ dx-\int_{I}\alpha\sqrt{1+\psi_{x}^2}\ dx\\
		& +\int_{(a,b)\setminus I} \tilde{\alpha}[u^{\varepsilon''}] \sqrt{1+(u_x^{\varepsilon''})^2}\ dx - \int_{(a,b)\setminus I} \tilde{\alpha}[\bar{u}] \sqrt{1+\bar{u}_x^2}\ dx.
	\end{align*}
	By Lemma \ref{lscgeneral}, we get
	$$\liminf_{\varepsilon''\to 0} \int_a^b\chi_{I} \sqrt{1+(u_{x}^{\varepsilon''})^2}\ dx \geq \int_a^b\chi_{I}\sqrt{1+\bar{u}_{x}^2}\ dx = \int_a^b\chi_{I}\sqrt{1+\psi_{x}^2}\ dx$$
	and
	$$\liminf_{\varepsilon''\to 0} \int_a^b \chi_{(a,b)\setminus I} \tilde{\alpha}[u^{\varepsilon''}] \sqrt{1+(u_x^{\varepsilon''})^2}\ dx \geq \int_a^b \chi_{(a,b)\setminus I} \tilde{\alpha}[\bar{u}] \sqrt{1+\bar{u}_x^2}\ dx.$$
	Therefore, by $0<\delta/3\leq\delta_{\psi,\alpha}$ and Lemma \ref{psidelta}, we have
	\begin{align*}
		\liminf_{\varepsilon''\to 0}\varepsilon''F_{\varepsilon''}[u^{\varepsilon''}] &\geq \int_{I}\sqrt{1+\psi_{x}^2}-\alpha\sqrt{1+\psi_{x}^2}\ dx\\
		&\geq \frac{\delta}{3}\left(\inf_{I}\sqrt{1+{\psi_x}^2}-\overline{\alpha}\sup_{I}\sqrt{1+{\psi_x}^2}\right)\\
		&\geq \frac{\delta}{3}\cdot\frac{1-\overline{\alpha}}{2} > 0,
	\end{align*}
	that is, $\liminf_{\varepsilon''\to 0}F_{\varepsilon''}[u^{\varepsilon''}]=\infty$.
	This is a contradiction.
\end{proof}

\begin{lemma}\label{redlem2}
	Let $u^{\varepsilon}\rightarrow\bar{u}$ in $X^{1,1}(a,b)$ as in Assumption \ref{assumpseq}.
	Define
	$$\delta_{\psi,\alpha,\bar{u}}:=\min\left\{\delta_{\psi,\alpha},\frac{\bar{x}_1-\bar{x}_0}{3},\ldots,\frac{\bar{x}_{N+1}-\bar{x}_{N}}{3}\right\}.$$
	Then for any $0<\delta\leq\delta_{\psi,\alpha,\bar{u}}$ there exists a sequence $\{v_\delta^{{\varepsilon}_j}\}_j\subset X^{2,1}(a,b)$ such that,
	\begin{enumerate}
		\item $v_\delta^{{\varepsilon}_j}\rightarrow\bar{u}$ in $W^{1,1}(a,b)$,
		\item $\liminf_{j\to\infty}F_{{\varepsilon}_j}[v_\delta^{{\varepsilon}_j}]\leq\liminf_{\varepsilon\to0}F_{\varepsilon}[u^\varepsilon]$,
		\item for any ${\varepsilon}_j>0$ and $0\leq i\leq N$ with $(\bar{x}_i,\bar{x}_{i+1})\subset(a,b)_0^{\bar{u}}$ the equality $v_\delta^{{\varepsilon}_j}\equiv\psi$ holds in $[\bar{x}_i+\delta,\bar{x}_{i+1}-\delta]$.
	\end{enumerate}
\end{lemma}

\begin{proof}
	Fix $0<\delta\leq\delta_{\psi,\alpha,\bar{u}}$ and denote the subsequence obtained in Lemma \ref{redlem1} by $\{u_\delta^{{\varepsilon}_j}\}_j$.
	Recall that, by Lemma \ref{redlem1}, for any $\varepsilon_j$ and $0\leq i\leq N$ with $(\bar{x}_i,\bar{x}_{i+1})\subset(a,b)_0^{\bar{u}}$ there exist $y_i^{\varepsilon_j}\in(\bar{x}_i,\bar{x}_i+\delta]$ and $y_{i+1}^{\varepsilon_j}\in[\bar{x}_{i+1}-\delta,\bar{x}_{i+1})$ such that $u_\delta^{{\varepsilon}_j}=\psi$ at $y_i^{\varepsilon_j}$ and $y_{i+1}^{\varepsilon_j}$.
	Moreover, since $u_\delta^{\varepsilon_j}$ is continuously differentiable and $u_\delta^{\varepsilon_j}\geq\psi$, we also have $(u_\delta^{\varepsilon_j})_x=\psi_x$ at $y_i^{\varepsilon_j}$ and $y_{i+1}^{\varepsilon_j}$. \par
	We define $v_\delta^{{\varepsilon}_j}$ by replacing $u_\delta^{\varepsilon_j}$ to be $\psi$ in $(y_{i}^{\varepsilon_j},y_{i+1}^{\varepsilon_j})$ for all $i$ such that $(\bar{x}_i,\bar{x}_{i+1})\subset(a,b)_0^{\bar{u}}$.
	Then $v_\delta^{{\varepsilon}_j}$ belongs to $X^{2,1}(a,b)$ since $u_\delta^{\varepsilon_j}=\psi$ and $(u_\delta^{\varepsilon_j})_x=\psi_x$ at all $y_i^{\varepsilon_j}$.
	Furthermore, the sequence $\{v_\delta^{{\varepsilon}_j}\}$ satisfies all the conditions in this lemma.
	Indeed, the first condition follows by
	$$\|v_\delta^{{\varepsilon}_j}-\bar{u}\|_{W^{1,1}}\leq\|u_\delta^{{\varepsilon}_j}-\bar{u}\|_{W^{1,1}},$$
	the second one by
	$$E_0[u_\delta^{{\varepsilon}_j}]-E_0[v_\delta^{{\varepsilon}_j}]=E_0[u_\delta^{{\varepsilon}_j}-v_\delta^{{\varepsilon}_j}+\bar{u}]-E_0[\bar{u}]\geq0$$
	and
	$$F_{{\varepsilon}_j}[u_\delta^{{\varepsilon}_j}]-F_{{\varepsilon}_j}[v_\delta^{{\varepsilon}_j}]\geq-\varepsilon_j\int_a^b\kappa_\psi^2\sqrt{1+\psi_x^2}\ dx\xrightarrow{j\to\infty}\geq0,$$
	and the third one by the definition of $v_\delta^{{\varepsilon}_j}$.
\end{proof}

\begin{lemma}\label{redlem3}
	Let $u^{\varepsilon}\rightarrow\bar{u}$ and $\delta_{\psi,\alpha,\bar{u}}$ as in Lemma \ref{redlem2}.
	Then for any $0<\delta\leq\frac{\delta_{\psi,\alpha,\bar{u}}}{2}$ there exists a sequence $\{w_\delta^{{\varepsilon}_j}\}_j\subset X^{2,1}(a,b)$ such that,
	\begin{enumerate}
		\item $\liminf_{j\to\infty}F_{{\varepsilon}_j}[w_\delta^{{\varepsilon}_j}]\leq\liminf_{\varepsilon\to0}F_{\varepsilon}[u^{\varepsilon}]$,
		\item there exists $M>0$ such that $w^{{\varepsilon}_j}_\delta$ is $\delta$-associated with $\bar{u}$ for any $j \geq M$.
	\end{enumerate}
\end{lemma}

\begin{proof}
	We first assume that $\bar{u}>\psi$ on $\partial(a,b)$.
	Fix $0<\delta\leq\frac{\delta_{\psi,\alpha,\bar{u}}}{2}$.
	Define
	$$I_{\delta_*}^i:=\left\{y\in (a,b) \mid |y-\bar{x}_{i}|\leq\delta_*\right\}\subset\subset(a,b)$$
	for $1\leq i\leq N$ and denote by $I_{\delta_*}$ the (disjoint) union of $I_{\delta_*}^i$, where $\delta_*<\delta$ is taken so that
	$$\sup\left\{|\bar{u}(x)-\psi(x)|\left|\ x\in I_{\delta_*} \right\}\right.\leq\delta.$$
	Take $\{v_{\delta_*}^{{\varepsilon}_j}\}_j$ as in Lemma \ref{redlem2} for $\delta_*$.
	Since
	$$\inf{\left\{\bar{u}(x)-\psi(x) \mid x\in(a,b)_+^{\bar{u}}\setminus I_{\delta_*}\right\}}>0$$
	and $v_{\delta_*}^{{\varepsilon}_j}\rightarrow\bar{u}$ as $j\to\infty$ uniformly (by Sobolev embedding), there exists $M>0$ such that $\|v_{\delta_*}^{\varepsilon_j}-\bar{u}\|_\infty\leq\delta$ and $v_\delta^{\varepsilon_j}>\psi$ in $(a,b)_+^{\bar{u}}\setminus I_{\delta_*}$ for any $j\geq M$.
	On the other hand, for any ${\varepsilon_j}>0$ we find $v_{\delta_*}^{{\varepsilon}_j}\equiv\psi$ in $(a,b)_0^{\bar{u}}\setminus I_{\delta_*}$ by the third condition of Lemma \ref{redlem2}. \par
	Therefore, to prove this lemma, we only have to reduce $v_{\delta_*}^{\epsilon_j}$ in $I_{\delta_*}$.
	Define
	$$y_i^{\varepsilon_j}:=\min\{y\in I_{\delta_*}^i \mid v_{\delta_*}^{\varepsilon_j}(y)=\psi(y)\}$$
	for $i$ with $(\bar{x}_i,\bar{x}_{i+1})\subset(a,b)_0^{\bar{u}}$.
	Then we find $v_{\delta_*}^{\varepsilon_j}=\psi$ on $\partial[y_i^{\varepsilon_j},\bar{x}_i+{\delta_*}]$.
	Thus, by Lemma \ref{reglem1}, the reduction replacing $v_{\delta_*}^{\varepsilon_j}$ to be $\psi$ only in $[y_i^{\varepsilon_j},\bar{x}_i+{\delta_*}]$ does not increase $E_0$ and retains the function in $X^{2,1}(a,b)$.
	Similarly, for $i$ with $(\bar{x}_i,\bar{x}_{i+1})\subset(a,b)_+^{\bar{u}}$, we can get the result reversed left and right.
	The consequence follows by defining $w_\delta^{{\varepsilon}_j}$ as the function reduced as above for all $1\leq i\leq N$.
	Indeed, the function $w_\delta^{{\varepsilon}_j}$ is $\delta$-associated with $\bar{u}$ for any $j\geq M$ by the definition of $w_\delta^{\varepsilon_j}$. 
	Here $S\bar{x}_i=y_i^{\varepsilon_j}$.
	Moreover, the first condition follows by for any $j\geq M$ we have $E_0[v_{\delta_*}^{{\varepsilon}_j}]\geq E_0[w_\delta^{{\varepsilon}_j}]$ and
	$$F_{\varepsilon_j}[v_{\delta_*}^{{\varepsilon}_j}]-F_{\varepsilon_j}[w_\delta^{{\varepsilon}_j}]\geq -\varepsilon_j\int_a^b\kappa_\psi^2\sqrt{1+\psi_x^2}\ dx\xrightarrow{j\to\infty}\geq0.$$\par
	Finally, we mention the case $\bar{u}=\psi$ at $\bar{x}_0(=a)$ or $\bar{x}_{N+1}(=b)$ or both.
	In addition to the above proof, we only need to reduce $v_{\delta_*}^{{\varepsilon}_j}$ near $\bar{x}_0$ or $\bar{x}_{N+1}$ or both.
	Let us only consider for $\bar{x}_0$ since it is similar for $\bar{x}_{N+1}$.
	When we take $\delta_*$, the similar condition for $i=0$ shall be added.
	If $\bar{u}(\bar{x}_0)=\psi(\bar{x}_0)$ and $(\bar{x}_0,\bar{x}_1)\subset(a,b)_0^{\bar{u}}$, then the consequence follows by replacing $v_{\delta_*}^{{\varepsilon}_j}$ by $\psi$ in $(\bar{x}_0,\bar{x}_0+{\delta_*}]$.
	If $\bar{u}(\bar{x}_0)=\psi(\bar{x}_0)$ and $(\bar{x}_0,\bar{x}_1)\subset(a,b)_+^{\bar{u}}$, then $v_{\delta_*}^{{\varepsilon}_j}$ may touch $\psi$ near $\bar{x}_0$ even if $j$ is sufficiently large.
	Nevertheless, the consequence follows by replacing $v_{\delta_*}^{{\varepsilon}_j}$ by $\psi$ in $(\bar{x}_0,y_0^{\varepsilon_j}]$, where
	$$y_0^{\varepsilon_j}:=\max\{y\in(\bar{x}_0,\bar{x}_0+{\delta_{*}}] \mid v_{\delta_*}^{\varepsilon_j}(y)=\psi(y)\}\ (=S\bar{x}_0).$$
	The proof is completed.
\end{proof}

\subsection{Lower estimate for geometric energies}\label{liminfsub3}

In this subsection, we give a lower estimate for the functions as obtained in \S\ref{liminfsub2}.
In Proposition \ref{lbprop2}, for functions $\delta$-associated with a minimizer of $E_0$, we rewrite the energy in order to consider geometrically, and obtain a key estimate in Proposition \ref{lbprop3}.

\begin{definition}[$W^{2,1}$-curve and geometric energies]\label{defcurve}\ 
	\begin{enumerate}
		\item We say that $\gamma:[0,1]\rightarrow\mathbb{R}^2$ is a {\it (regular) $W^{2,1}$-curve} if the two components $\gamma_1$ and $\gamma_2$ belong to $W^{2,1}(0,1)\subset C^1([0,1])$ and
		$$|\dot{\gamma}|=\sqrt{\langle\dot{\gamma},\dot{\gamma}\rangle}=\sqrt{\gamma_1'^2+\gamma_2'^2}>0\ \text{in }[0,1].$$
		\item For a $W^{2,1}$-curve $\gamma$, we denote by $\bar{\gamma}$ the geodesic (segment) from $\gamma(0)$ to $\gamma(1)$.
		\item For a $W^{2,1}$-curve $\gamma$, we denote the tension (length) by
		$$\mathcal{L}[\gamma]:=\int_0^1|\dot{\gamma}(t)|dt.$$
		\item For a $W^{2,1}$-curve $\gamma$, we denote the bending energy by
		$$\mathcal{B}[\gamma]:=\int_0^1|\kappa_\gamma(t)|^2|\dot{\gamma}(t)|dt,$$
		where $\kappa_\gamma$ is the curvature of $\gamma$ defined by $\kappa_\gamma:=|\dot{\gamma}|^{-3}|\gamma_1'\gamma_2''-\gamma_2'\gamma_1''|$.
		\item For a $W^{2,1}$-curve $\gamma$, we define the {\it boundary warp energy} by 
		$$\mathcal{W}[\gamma]:=4\left(2\sqrt{2}-\sqrt{1+\cos{\theta_\gamma^0}}-\sqrt{1+\cos{\theta_\gamma^1}}\right),$$
		where $\theta_\gamma^0$ and $\theta_\gamma^1$ are the {\it boundary warp angles} of $\gamma$ defined by
		$$\angle\gamma\bar{\gamma}:=\arccos{(\langle\dot{\gamma},\dot{\bar{\gamma}}\rangle/|\dot{\gamma}||\dot{\bar{\gamma}}|)}\in[0,\pi]$$
		at $\gamma(0)$ and $\gamma(1)$ respectively.
	\end{enumerate}
\end{definition}

\begin{remark}\ 
	$\mathcal{L}$, $\mathcal{B}$, $\mathcal{W}$, $\theta_\gamma^0$ and $\theta_\gamma^1$ are well-defined for $W^{2,1}$-curves, i.e. they are invariant by $W^{2,1}$-reparameterization.
	They are also invariant with respect to translation, reflection and rotation.
	To be more precise, see Appendix.
	Moreover, for any $u\in W^{2,1}$ on a bounded interval, the graph of $u$ is a $W^{2,1}$-curve.
\end{remark}

\begin{proposition}\label{lbprop2}
	Let $\bar{u}\in{\rm argmin}_{X^{1,1}}E_0$ be $[\bar{x}_0;\ldots;\bar{x}_{N+1}]$-regular and $m\in\mathbb{N}$ be the number of the connected components of $(a,b)_+^{\bar{u}}$.
	Then there exists $\delta(\psi,\bar{u})>0$ such that for any $\varepsilon>0$, $0<\delta\leq\delta(\psi,\bar{u})$ and $u\in X^{2,1}(a,b)$ which is $\delta$-associated with $\bar{u}$ the inequality holds:
	\begin{eqnarray}\label{lbrep1}
		F_\varepsilon[u]\geq\sum_{k=1}^{m}\varepsilon\mathcal{B}[\gamma^k]+\frac{1}{\varepsilon}(\mathcal{L}[\gamma^k]-\mathcal{L}[\bar{\gamma}^k]),
	\end{eqnarray}
	where $\gamma^k$ $(1\leq k\leq m)$ is a $W^{2,1}$-curve which is the graph of $u$ on the $k$-th connected component of $(a,b)_+^{u}$.\par
	Moreover, the difference between the boundary warp angle of $\gamma^k$ at an endpoint and the contact angle of $\bar{u}$ with $\psi$ at the corresponding endpoint of the $k$-th connected component of $(a,b)_+^{\bar{u}}$ tends to be zero as $\delta\downarrow0$ independently of $u$.
\end{proposition}

\begin{proof}
	Fix $\delta>0$ and $u\in W^{2,1}(a,b)$ which is $\delta$-associated with $\bar{u}$.
	We denote $x_i=S\bar{x}_i$ for $0\leq i\leq N+1$ (the shift operator $S$ is defined in Definition \ref{deltaasso}) and define $u_\delta\in W^{1,1}(a,b)$ by
	\begin{align*}
		u_\delta(x):=
		\begin{cases}
			\frac{u(x_{i+1})-u(x_i)}{x_{i+1}-x_i}(x-x_i)+u(x_i) & (x\in(x_i,x_{i+1})\subset(a,b)_+^{u}),\\
			\psi(x) & (\text{otherwise}).
		\end{cases}
	\end{align*}
	Then, since any contact angle of $\bar{u}$ with $\psi$ is positive by Theorem \ref{regezero}, there exists $\delta(\psi,\bar{u})>0$ such that $u_\delta$ is in $X^{1,1}(a,b)$ for any $0<\delta\leq\delta(\psi,\bar{u})$.
	Therefore, since $E_0[u_\delta]\geq E_0[\bar{u}]$ by the minimality of $\bar{u}$, we have
	$$F_{\varepsilon}[u]\geq \varepsilon\int_{a}^{b} {\kappa_u}^2\sqrt{1+u_x^2}\ dx+\frac{1}{\varepsilon}(E_0[u]-E_0[u_\delta]).$$
	Moreover, we can restrict the domain of integration of the right-hand term to the disjoint union of $(x_i,x_{i+1})\subset(a,b)_+^{u}$ thus we have
	\begin{multline*}
		F_\varepsilon[u]\geq\sum_{\substack{0\leq i\leq N\\ (x_i,x_{i+1})\subset(a,b)_+^u}}\left(\varepsilon\int_{x_i}^{x_{i+1}} \kappa_u^2\sqrt{1+u_x^2}\ dx\right.\\ 
		\left.+\frac{1}{\varepsilon}\int_{x_i}^{x_{i+1}} \sqrt{1+u_x^2} - \sqrt{1+\tan^2{\tilde{\theta}_i}}\ dx \right),
	\end{multline*}
	where $\tilde{\theta}_i:=\arctan{\left(\frac{u(x_{i+1})-u(x_{i})}{x_{i+1}-x_i}\right)}$.\par
	By the above argument, if we take $\gamma^k$ as the statement then (\ref{lbrep1}) holds.
	Now we take any $1\leq i\leq N$ with $(x_i,x_{i+1})\subset(a,b)_+^u$.
	The boundary warp angle of $\gamma^k$ at $(x_i,u(x_i))$ is $\tilde{\theta}_i-\theta_u(x_i)=\tilde{\theta}_i-\theta_\psi(x_i)$, that is,
	$$\arctan{\left(\frac{u(x_{i+1})-u(x_{i})}{x_{i+1}-x_i}\right)}-\theta_\psi(x_i),$$
	and the contact angle of $\bar{u}$ at $\bar{x}_i$ is $\theta_{\bar{u}}(\bar{x}_i+)-\theta_\psi(\bar{x}_i)$, that is,
	$$\arctan{\left(\frac{\bar{u}(\bar{x}_{i+1})-\bar{u}(\bar{x}_{i})}{\bar{x}_{i+1}-\bar{x}_i}\right)}-\theta_\psi(\bar{x}_i). $$
	Therefore, the difference of them tends to be zero as $\delta\downarrow0$ not depending on $u$ (but only $\psi$ and $\bar{u}$) since $\|u-\bar{u}\|_\infty\leq\delta$, $|x_i-\bar{x}_i|\leq\delta$, and the functions $\bar{u}$, $\arctan(\cdot)$ and $\theta_\psi$ are uniformly continuous.
	We can similarly consider for $1\leq i\leq N$ with $(x_i,x_{i+1})\subset(a,b)_0^u$, thus the proof is completed.
\end{proof}

We now obtain a key estimate for the right-hand term of (\ref{lbrep1}).

\begin{proposition}\label{lbprop3}
	For any $\varepsilon>0$ and $W^{2,1}$-curve $\gamma$ with $\theta_\gamma^0,\theta_\gamma^1\in[0,\pi/2)$, the inequality holds:
	$$\varepsilon\mathcal{B}[\gamma]+\frac{1}{\varepsilon}(\mathcal{L}[\gamma]-\mathcal{L}[\bar{\gamma}])\geq\mathcal{W}[\gamma].$$
\end{proposition}

This follows by Lemma \ref{redlem5} and Lemma \ref{lblem1}.

\begin{lemma}\label{redlem5}
	Let $\varepsilon>0$.
	For any $W^{2,1}$-curve $\gamma$ with $\theta_\gamma^0,\theta_\gamma^1\in[0,\pi/2)$, there exist $L_0,L_1>0$ and a function $\hat{u}\in W^{2,1}(-L_0,L_1)$ such that $\hat{u}_x(0)=0$, $\hat{u}_x(-L_0)=\tan\theta_\gamma^0$, $\hat{u}_x(L_1)=\tan\theta_\gamma^1$ and
	\begin{align*}
		\varepsilon\mathcal{B}[\gamma]+\frac{1}{\varepsilon}(\mathcal{L}[\gamma]-\mathcal{L}[\bar{\gamma}]) \geq \int_{-L_0}^{L_1}\varepsilon \kappa_{\hat{u}}^2\sqrt{1+{\hat{u}_x}^2}+ \frac{1}{\varepsilon}\left(\sqrt{1+{\hat{u}_x}^2} - 1\right)dx.
	\end{align*}
\end{lemma}

\begin{proof}
	We can assume $\mathcal{B}[\gamma]<\infty$, $\gamma(0)=(0,0)$, $\gamma(1)=(\gamma_1(1),0)$ with $\gamma_1(1)>0$ without loss of generality.
	Note that
	$$\mathcal{L}[\bar{\gamma}]=\gamma_1(1),\quad\tan\theta_\gamma^0=|\gamma_2'(0)|/\gamma_1'(0),\quad\tan\theta_\gamma^1=|\gamma_2'(1)|/\gamma_1'(1),$$
	and there exists $0<t_\gamma<1$ such that $\gamma_2'(t_\gamma)=0$.
	Now we fix $\gamma:[0,1]\rightarrow\mathbb{R}^2$.\par
	Take an arbitrary angle $\theta_\gamma^{*}\in(\max\{\theta_\gamma^0,\theta_\gamma^1,\pi/4\},\pi/2)$ and for $t\in[0,1]$ define
	\begin{align*}
		\hat{\gamma}(t):=\int_0^t \left(
		\begin{array}{c}
			\max\{|\gamma_1'|(r),(\cos{2\theta_\gamma^{*}})|\gamma_1'|(r)+(\sin{2\theta_\gamma^{*}})|\gamma_2'|(r)\}\\
			\min\{|\gamma_2'|(r),(\sin{2\theta_\gamma^{*}})|\gamma_1'|(r)-(\cos{2\theta_\gamma^{*}})|\gamma_2'|(r)\}
		\end{array}
		\right)dr.
	\end{align*}
	Then $\hat{\gamma}$ is a $W^{2,1}$-curve such that $\hat{\gamma}$ is in the first quadrant of $\mathbb{R}^2$,
	$$\dot{\hat{\gamma}}/|\dot{\hat{\gamma}}|\in\{(\cos{\theta},\sin{\theta})\in\mathbb{S}^1|\ 0\leq\theta\leq\theta_\gamma^{*}\}\quad\text{in }[0,1],$$
	and the following conditions hold:
	$$\mathcal{B}[\hat{\gamma}]=\mathcal{B}[\gamma],\quad\mathcal{L}[\hat{\gamma}]=\mathcal{L}[\gamma],\quad\hat{\gamma}(0)=(0,0),\quad\hat{\gamma}_1(1)\geq\gamma_1(1),\quad\hat{\gamma}_2'(t_\gamma)=0,$$
	$$\hat{\gamma}_2'(0)/\hat{\gamma}'_1(0)=|\gamma_2'(0)|/\gamma_1'(0),\quad\hat{\gamma}_2'(1)/\hat{\gamma}'_1(1)=|\gamma_2'(1)|/\gamma_1'(1).$$
	These conditions hold since we only used translation or reflection partially in this transformation, and while
	$$\dot{\gamma}/|\dot{\gamma}|\in\{(\cos{\theta},\sin{\theta})\in\mathbb{S}^1|\ |\theta|\leq\theta_\gamma^*\}$$
	only translation or one reflection $(\gamma_1,\gamma_2)\to(\gamma_1,-\gamma_2)$.
	In addition, since $\hat{\gamma}_1'>0$ in $[0,1]$, we can define the inverse function of $\hat{\gamma}_1$ and it is in $W^{2,1}(0,\hat{\gamma}_1(1))$. Thus, by taking
	$$\hat{u}(x):=\hat{\gamma}_2\circ\hat{\gamma}_1^{-1}(x+\hat{\gamma}_1(t_\gamma)),\quad L_0:=\hat{\gamma}_1(t_\gamma), \quad L_1:=\hat{\gamma}_1(1)-\hat{\gamma}_1(t_\gamma),$$
	we obtain a desired function.
	Indeed, we have
	$$\hat{u}_x(0)=(\hat{\gamma}_2\circ\hat{\gamma}_1^{-1})'(\hat{\gamma}_1(t_\gamma))=\hat{\gamma}'_2(t_\gamma)/\hat{\gamma}'_1(t_\gamma)=0,$$
	$$\hat{u}_x(-L_0)=(\hat{\gamma}_2\circ\hat{\gamma}_1^{-1})'(0)=\hat{\gamma}'_2(0)/\hat{\gamma}'_1(0)=|\gamma_2'(0)|/\gamma_1'(0)=\tan{\theta_\gamma^0}$$
	and $\hat{u}_x(L_1)=\tan{\theta_\gamma^1}$ similarly and the desired inequality follows by
	$$\mathcal{B}[\gamma]=\mathcal{B}[\hat{\gamma}]=\int\kappa_{\hat{u}}^2\sqrt{1+{\hat{u}_x}^2},\quad\mathcal{L}[\gamma]=\mathcal{L}[\hat{\gamma}]=\int\sqrt{1+{\hat{u}_x}^2},$$
	$$\int_{-L_0}^{L_1}dx=L_0+L_1=\hat{\gamma}_1(1)\geq\gamma_1(1)=\mathcal{L}[\bar{\gamma}].$$
	The proof is completed.
\end{proof}

\begin{lemma}\label{lblem1}
	Let $v\in W^{2,1}(0,L)$ and $\theta\in[0,\pi/2)$ satisfying $v_x(0)=0$ and $|v_x(L)|=\tan\theta$.
	Then the inequality holds:
	$$\int_{0}^{L} \varepsilon\kappa_v^2\sqrt{1+{v_x}^2} + \frac{1}{\varepsilon}\left(\sqrt{1+{v_x}^2} - 1\right)dx\geq4\left(\sqrt{2}-\sqrt{1+\cos{\theta}}\right).$$
\end{lemma}

\begin{proof}
	For such $v\in W^{2,1}(0,L)$, we have
	\begin{align*}
		&\int_{0}^{L} \varepsilon{\left(\frac{v_{xx}}{(1+v_x^2)^{3/2}}\right)}^2\sqrt{1+{v_x}^2} + \frac{1}{\varepsilon} \left(\sqrt{1+{v_x}^2} - 1\right)dx \\
		&\geq\int_0^L 2\frac{|v_{xx}|}{(1+v_x^2)^{5/4}}\sqrt{\sqrt{1+{v_x}^2} - 1}\ dx\\
		&\geq \left|\int_0^Lf'(v_x)v_{xx}\ dx\right| = |f(v_x(L))-f(v_x(0))| = |f(v_x(L))|,
	\end{align*}
	where $f$ is the odd function given by
	\begin{eqnarray}\label{evenfunc}
		f(y):=\int_0^y \frac{2\sqrt{\sqrt{1+{z}^2} - 1}}{(1+z^2)^{5/4}}\ dz.
	\end{eqnarray}
	The first inequality follows by the easy trick $\varepsilon X^2 + \varepsilon^{-1}Y^2 \geq 2XY$.
	Moreover, we have $|f(v_x(L))|=f(|v_x(L)|)=f(\tan\theta)$ and
	\begin{align*}
		f(\tan{\theta}) &= 2\int_0^{\tan{\theta}} \frac{\sqrt{\sqrt{1+{z}^2} - 1}}{(1+z^2)^{5/4}}\ dz\\
		&= 2\int_0^{\theta} \sqrt{1-\cos{\varphi}}\ d\varphi\quad(z=\tan{\varphi}) \\
		&= 2\int_{\cos{\theta}}^1 \frac{dw}{\sqrt{1+w}}\quad(w=\cos{\varphi}) \\
		&= 4\left(\sqrt{2}-\sqrt{1+\cos{\theta}}\right),
	\end{align*}
	thus the proof is completed.
\end{proof}

\begin{proof}[Proposition \ref{lbprop3}]
	Let $\gamma$ be any $W^{2,1}$-curve and $\hat{u}$ be the function obtained by Lemma \ref{redlem5}.
	Using Lemma \ref{lblem1} for the $W^{2,1}$-functions $\hat{u}(-x)$ on $[0,L_0]$ and $\hat{u}(x)$ on $[0,L_1]$, then we have
	$$\varepsilon\mathcal{B}[\gamma]+\frac{1}{\varepsilon}(\mathcal{L}[\gamma]-\mathcal{L}[\bar{\gamma}])\geq4\left(\sqrt{2}-\sqrt{1+\cos\theta_\gamma^0}\right)+4\left(\sqrt{2}-\sqrt{1+\cos\theta_\gamma^1}\right)$$
	and the right-hand term is nothing but $\mathcal{W}[\gamma]$.
\end{proof}

\subsection{Completion of the proof of the liminf condition}\label{liminfsub4}
We shall prove the liminf condition of Theorem \ref{gammathm}.
\begin{proof}
	Take any sequence $u^{\varepsilon}\to u$ in $X^{1,1}(a,b)$ and fix it.
	We can assume Assumption \ref{assumpseq} by Proposition \ref{restseq} (thus $u=\bar{u}$ is $[\bar{x}_0;\ldots;\bar{x}_{N+1}]$-regular).
	Then for any sufficiently small $\delta>0$, we have
	$$\liminf_{\varepsilon\to 0}F_{\varepsilon}[u^{\varepsilon}]\geq\sum_{i=1}^N 4\left(\sqrt{2}-\sqrt{1+\cos{\theta_{i,\delta}}}\right)$$
	by Proposition \ref{lbprop1}, \ref{lbprop2} and \ref{lbprop3}, where $\theta_{i,\delta}\in(0,\pi/2)$ are angles satisfying $\theta_{i,\delta}\rightarrow\theta_i$ as $\delta\downarrow0$ and $\theta_i$ is the contact angle of $\bar{u}$ with $\psi$ at $\bar{x}_i$.
	Thus we get the consequence by taking $\delta\downarrow 0$ since $\cos\theta_i=\alpha(\bar{x}_i)$ and (\ref{repsinglim}) hold.
\end{proof}

\section{Lim-sup condition of $\Gamma$-convergence} \label{limsupsec}

Finally, we prove the limsup condition of Theorem \ref{gammathm}.
\begin{proof}
	We construct sequences concretely by modifying singularities of minimizers of $E_0$.
	To this end, for arbitrary $\varepsilon>0$ and $0<\theta<\pi/2$ we consider the following ODE:
	\begin{eqnarray}\label{odeeps}
		\begin{cases}
			U''=\displaystyle\frac{1}{\varepsilon}(1+(U')^2)^{5/4}\sqrt{\sqrt{1+(U')^2}-1}\quad\text{in }(-\infty,0],\\
			U(0)=0,\ U'(0)=\tan{\theta}>0.
		\end{cases}
	\end{eqnarray}
	This Cauchy problem has the unique solution $U_{\theta,\varepsilon}:(-\infty,0]\rightarrow\mathbb{R}$ satisfying $U_{\theta,\varepsilon}<0$, $U_{\theta,\varepsilon}'>0$ and $U_{\theta,\varepsilon}''>0$ in $(-\infty,0)$.
	Moreover, $U_{\theta,\varepsilon}(x)=\varepsilon U_{\theta,1}(x/\varepsilon)$ for $x\in(-\infty,0)$, $\lim_{x\to-\infty}U_{\theta,1}'(x)=0$ and by definition we see
	\begin{align*}
		& \varepsilon\int_{-\varepsilon^{2/3}}^{0} {\left(\frac{U_{\theta,\varepsilon}''}{(1+(U_{\theta,\varepsilon}')^2)^{3/2}}\right)}^2\sqrt{1+(U_{\theta,\varepsilon}')^2}\ dx\\
		& +\frac{1}{\varepsilon}\int_{-\varepsilon^{2/3}}^{0} \left(\sqrt{1+(U_{\theta,\varepsilon}')^2} - 1\right)dx\\
		=& \int_{-\varepsilon^{-1/3}}^{0} {\left(\frac{U_{\theta,1}''}{(1+(U_{\theta,1}')^2)^{3/2}}\right)}^2\sqrt{1+(U_{\theta,1}')^2}\ dx\\
		& +\int_{-\varepsilon^{-1/3}}^{0} \left(\sqrt{1+(U_{\theta,1}')^2} - 1\right)dx\\
		=& \int_{-\varepsilon^{-1/3}}^{0} 2\frac{U_{\theta,1}''}{(1+(U_{\theta,1}')^2)^{5/4}}\sqrt{\sqrt{1+(U_{\theta,1}')^2} - 1}\ dx\\
		=& f(U_{\theta,1}'(0))-f(U_{\theta,1}'(-\varepsilon^{-1/3})) \xrightarrow{\varepsilon\to0} f(\tan{\theta})=4\left(\sqrt{2}-\sqrt{1+\cos{\theta}}\right),
	\end{align*}
	where $f$ is defined in (\ref{evenfunc}).\par
	Considering the above, for any two points in $\mathbb{R}^2$ and angles $\theta,\theta'\in(0,\pi/2)$, we can take a sequence of $W^{2,1}$-curves $\{\gamma^\varepsilon_{\theta,\theta'}\}_{\varepsilon>0}$ connecting the two points such that the boundary warp angles are $\theta$, $\theta'$ and the (rescaled) total energy of the curves converges to the boundary warp energy of them, that is,
	\begin{align*}
		\lim_{\varepsilon\to0}\left(\varepsilon\mathcal{B}[\gamma^\varepsilon_{\theta,\theta'}]+\frac{1}{\varepsilon}(\mathcal{L}[\gamma^\varepsilon_{\theta,\theta'}]-\mathcal{L}[\bar{\gamma}^\varepsilon_{\theta,\theta'}])\right)\\
		=4\left(2\sqrt{2}-\sqrt{1+\cos{\theta}}-\sqrt{1+\cos{\theta'}}\right).
	\end{align*}
	This is achieved, for example, by rotating the graph of a $W^{2,1}$-function $V_{\varepsilon,\theta,\theta'}$ on some interval $[A,B]$ defined by
	\begin{align*}
		V_{\varepsilon,\theta,\theta'}(x):=
		\begin{cases}
			U_{\theta,\varepsilon}(A-x)\quad x\in[A,A+\varepsilon^{2/3}],\\
			\text{suitable function}\quad\text{in }(A+\varepsilon^{2/3},B-\varepsilon^{2/3}),\\
			U_{\theta',\varepsilon}(x-B)\quad x\in[B-\varepsilon^{2/3},B].
		\end{cases}
	\end{align*}
	In $[A,A+\varepsilon^{2/3}]\cup[B-\varepsilon^{2/3},B]$ the total energy converges to
	$$4\left(2\sqrt{2}-\sqrt{1+\cos{\theta}}-\sqrt{1+\cos{\theta'}}\right)$$
	as $\varepsilon\downarrow0$.
	In $(A+\varepsilon^{2/3},B-\varepsilon^{2/3})$, for example we use two arcs of the circle with radius $\varepsilon$ and central angles $\theta_\varepsilon$, $\theta'_\varepsilon$ tending to zero as $\varepsilon\downarrow0$ and a segment suitably (as Figure \ref{upperfig}).
	We decompose $(A+\varepsilon^{2/3},B-\varepsilon^{2/3})$ into the circular parts $I_c^\varepsilon\subset[A,B]$ and the segment part $I_s^\varepsilon\subset[A,B]$.
	Then we see that the total energy of the circular parts tends to be zero by
	\begin{align*}
		\varepsilon\int_{I_c^\varepsilon}\kappa^2 ds+\frac{1}{\varepsilon}\left(\int_{I_c^\varepsilon}ds-|{I_c^\varepsilon}|\right)&=\left(\varepsilon\cdot\frac{1}{\varepsilon^2}+\frac{1}{\varepsilon}\right)\int_{I_c^\varepsilon}ds-\frac{|I_c^\varepsilon|}{\varepsilon}\\
		&\approx O(\theta_\varepsilon)+O(\theta'_\varepsilon)\xrightarrow{\varepsilon\to0}0.
	\end{align*}
	The segment part also tends to be zero by
	\begin{align*}
		\varepsilon\int_{I_s^\varepsilon}\kappa^2 ds+\frac{1}{\varepsilon}\left(\int_{I_s^\varepsilon}ds-|I_s^\varepsilon|\right)&=\frac{1}{\varepsilon}\left(\int_{I_s^\varepsilon} ds-|I_s^\varepsilon|\right)\\
		&\leq\frac{\sqrt{|I_s^\varepsilon|^2+(O(\varepsilon^{2/3}))^2}-|I_s^\varepsilon|}{\varepsilon}\\
		&\approx O(\varepsilon^{1/3})\xrightarrow{\varepsilon\to0}0,
	\end{align*}
	since $U_{\theta,\varepsilon}(-\varepsilon^{2/3})\approx U_{\theta',\varepsilon}(-\varepsilon^{2/3})\approx O(\varepsilon^{2/3})$ as $\varepsilon\downarrow0$.\par
	Therefore, for any $\bar{u}\in\text{argmin}_{X^{1,1}}E_0$ which is $[\bar{x}_0,\ldots,\bar{x}_{N+1}]$-regular, by modifying it as above in $(\bar{x}_i,\bar{x}_{i+1})\subset(a,b)_+^{\bar{u}}$, we can take $\{u^\varepsilon\}_\varepsilon\subset X^{2,1}$ such that $u^\varepsilon\rightarrow\bar{u}$ in $X^{1,1}$ and
	$$\lim_{\varepsilon\to0} F_\varepsilon[u^\varepsilon]=\sum_{i=1}^N 4\left(\sqrt{2}-\sqrt{1+\cos\theta_i}\right)=\sum_{i=1}^N 4\left(\sqrt{2}-\sqrt{1+\alpha(\bar{x}_i)}\right),$$
	where $\theta_i$ is the contact angle of $\bar{u}$ with $\psi$ at $\bar{x}_i$.
	It is not trivial that the obtained functions $u^\varepsilon$ belong to $X^{2,1}$ thus we should be carefully.
	Note that the curves obtained by the modification can be represented as the graph of a $W^{2,1}$-function if $\varepsilon>0$ is sufficiently small.
	Thus $\{u^\varepsilon\}\subset W^{2,1}$.
	In addition, the second derivative of the modified parts of $u^\varepsilon$ goes to infinity as $\varepsilon\downarrow0$ near the free boundary of $\bar{u}$, thus we see $u^\varepsilon\geq\psi$ for any sufficiently small $\varepsilon>0$.
	This implies $\{u^\varepsilon\}\subset X^{2,1}$, and the proof is completed.
\end{proof}

\begin{figure}[htbp]
	\begin{center}
		\def\svgwidth{120mm}
		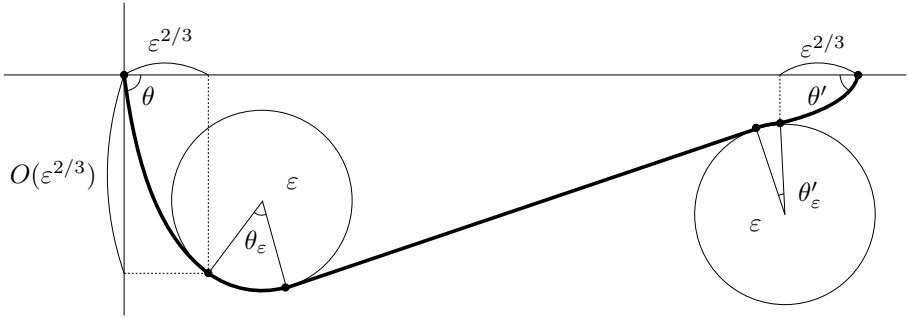
		\caption{construction of $V_{\varepsilon,\theta,\theta'}$}
		\label{upperfig}
	\end{center}
\end{figure}

\renewcommand{\thesection}{A}
\section{Appendix}\label{appendix}

\begin{lemma}[Change of variables]\label{cvlem1}
	Let $F$ be an $L^1$-integrable or non-negative measurable function on $[a,b]$ and $\Phi$ be a $C^1$-diffeomorphism from $[A,B]$ to $[a,b]$. Then 
	\begin{align*}
		\int_{a}^{b} F(x)\ dx=\int_{A}^{B} F(\Phi(y))|\Phi'(y)|\ dy. \label{chanvar1}
	\end{align*}
\end{lemma}

\begin{proof}
	The case that $\|F\|_{L^1}$ is finite follows by \cite[\S 3.3.3, Theorem 2]{EvGa92}.
	If $\|F\|_{L^1}$ is infinite then so is the right hand term since $\Phi$ is a $C^1$-diffeomorphism.
\end{proof}

\begin{lemma}\label{cvlem2}
	Let $u\in W^{1,1}(a,b)$ and $\Phi$ be a $C^1$-diffeomorphism from $[A,B]$ to $[a,b]$.
	Then $u\circ\Phi\in W^{1,1}(A,B)$ and
	$$(u\circ\Phi)'=(u'\circ\Phi)\Phi'.$$
\end{lemma}

\begin{proof}
	By Lemma \ref{cvlem1}.
\end{proof}

The following is nothing but \cite[Corollary 8.11]{Br11}.

\begin{lemma}\label{cvlem3}
	Let $u\in W^{1,1}(a,b)$ and $G\in C^1(\mathbb{R})$. Then $G\circ u \in W^{1,1}(a,b)$ and
	$$(G\circ u)'=(G'\circ u)u'.$$
\end{lemma}

The above lemmas lead to Lemma \ref{cvlem4}, further Lemma \ref{cvlem5} and \ref{cvlem6}.

\begin{lemma}\label{cvlem4}
	Let $u\in W^{2,1}(a,b)$, $\Phi\in W^{2,1}(A,B)$, $|\Phi'|>0$ in $[A,B]$, $\Phi([A,B])=[a,b]$ and $\Psi$ be the inverse function of $\Phi$. Then $u\circ\Phi\in W^{2,1}(A,B)$ and $\Psi\in W^{2,1}(a,b)$.
	Moreover,
	$$(u\circ\Phi)'=(u'\circ\Phi)\Phi',\ (u\circ\Phi)''=(u''\circ\Phi)(\Phi')^2+(u'\circ\Phi)\Phi'',$$
	\begin{align*}
		\Psi'=\frac{1}{\Phi'\circ\Psi},\ \Psi''=-\frac{\Phi''\circ\Psi}{(\Phi'\circ\Psi)^3}.
	\end{align*}
\end{lemma}

\begin{lemma}\label{cvlem5}
	Let $\gamma$ be a $W^{2,1}$-curve and $\Phi$ be a $W^{2,1}$-reparameterization, i.e. $\Phi\in W^{2,1}(0,1)$, $|\Phi'|>0$ in $[0,1]$ and $\Phi([0,1])=[0,1]$.
	Then $\gamma\circ\Phi$ is a $W^{2,1}$-curve.
	Moreover, $\mathcal{B}$, $\mathcal{L}$ and $\mathcal{W}$ are invariant by $W^{2,1}$-reparameterization.
\end{lemma}

\begin{lemma}\label{cvlem6}
	The energies $\mathcal{B}$, $\mathcal{L}$ and $\mathcal{W}$ are invariant with respect to translation, reflection and rotation.
\end{lemma}

\section*{Acknowledgements}
The author would like to thank his supervisor Yoshikazu Giga for suggesting this problem and fruitful discussions.
The author also would like to thank Oliver Pierre-Louis, who is one of the physicists considering this model, for his useful comments and remarks from a physical viewpoint.
This work was supported by a Grant-in-Aid for JSPS Fellows 15J05166 and the Program for Leading Graduate Schools, MEXT, Japan.

\end{document}

%% file: deltaS.eps_tex
\begingroup%
  \makeatletter%
  \providecommand\color[2][]{%
    \errmessage{(Inkscape) Color is used for the text in Inkscape, but the package 'color.sty' is not loaded}%
    \renewcommand\color[2][]{}%
  }%
  \providecommand\transparent[1]{%
    \errmessage{(Inkscape) Transparency is used (non-zero) for the text in Inkscape, but the package 'transparent.sty' is not loaded}%
    \renewcommand\transparent[1]{}%
  }%
  \providecommand\rotatebox[2]{#2}%
  \ifx\svgwidth\undefined%
    \setlength{\unitlength}{295.98125bp}%
    \ifx\svgscale\undefined%
      \relax%
    \else%
      \setlength{\unitlength}{\unitlength * \real{\svgscale}}%
    \fi%
  \else%
    \setlength{\unitlength}{\svgwidth}%
  \fi%
  \global\let\svgwidth\undefined%
  \global\let\svgscale\undefined%
  \makeatother%
  \begin{picture}(1,0.76197394)%
    \put(0,0){\includegraphics[width=\unitlength]{deltaS.eps}}%
    \put(0.72977596,0.71626159){\color[rgb]{0,0,0}\makebox(0,0)[lb]{\smash{$\delta$}}}%
    \put(0.82437654,0.60814663){\color[rgb]{0,0,0}\makebox(0,0)[lb]{\smash{$\delta$}}}%
    \put(0.40543109,0.02702874){\color[rgb]{0,0,0}\makebox(0,0)[lb]{\smash{$(a,b)_+^u$}}}%
    \put(0,0.39191672){\color[rgb]{0,0,0}\makebox(0,0)[lb]{\smash{$(a,b)_+^v$}}}%
    \put(0.83789091,0.43245983){\color[rgb]{0,0,0}\makebox(0,0)[lb]{\smash{$S$}}}%
  \end{picture}%
\endgroup%

%% file: deltau.eps_tex
\begingroup%
  \makeatletter%
  \providecommand\color[2][]{%
    \errmessage{(Inkscape) Color is used for the text in Inkscape, but the package 'color.sty' is not loaded}%
    \renewcommand\color[2][]{}%
  }%
  \providecommand\transparent[1]{%
    \errmessage{(Inkscape) Transparency is used (non-zero) for the text in Inkscape, but the package 'transparent.sty' is not loaded}%
    \renewcommand\transparent[1]{}%
  }%
  \providecommand\rotatebox[2]{#2}%
  \ifx\svgwidth\undefined%
    \setlength{\unitlength}{201.53bp}%
    \ifx\svgscale\undefined%
      \relax%
    \else%
      \setlength{\unitlength}{\unitlength * \real{\svgscale}}%
    \fi%
  \else%
    \setlength{\unitlength}{\svgwidth}%
  \fi%
  \global\let\svgwidth\undefined%
  \global\let\svgscale\undefined%
  \makeatother%
  \begin{picture}(1,0.79392646)%
    \put(0,0){\includegraphics[width=\unitlength]{deltau.eps}}%
    \put(0.27787426,0.19848162){\color[rgb]{0,0,0}\makebox(0,0)[lb]{\smash{$\psi$}}}%
    \put(0.59544485,0.71453382){\color[rgb]{0,0,0}\makebox(0,0)[lb]{\smash{$u$}}}%
    \put(0.03969632,0.35726691){\color[rgb]{0,0,0}\makebox(0,0)[lb]{\smash{$\delta$}}}%
  \end{picture}%
\endgroup%

%% file: deltav.eps_tex
\begingroup%
  \makeatletter%
  \providecommand\color[2][]{%
    \errmessage{(Inkscape) Color is used for the text in Inkscape, but the package 'color.sty' is not loaded}%
    \renewcommand\color[2][]{}%
  }%
  \providecommand\transparent[1]{%
    \errmessage{(Inkscape) Transparency is used (non-zero) for the text in Inkscape, but the package 'transparent.sty' is not loaded}%
    \renewcommand\transparent[1]{}%
  }%
  \providecommand\rotatebox[2]{#2}%
  \ifx\svgwidth\undefined%
    \setlength{\unitlength}{201.53bp}%
    \ifx\svgscale\undefined%
      \relax%
    \else%
      \setlength{\unitlength}{\unitlength * \real{\svgscale}}%
    \fi%
  \else%
    \setlength{\unitlength}{\svgwidth}%
  \fi%
  \global\let\svgwidth\undefined%
  \global\let\svgscale\undefined%
  \makeatother%
  \begin{picture}(1,0.79392646)%
    \put(0,0){\includegraphics[width=\unitlength]{deltav.eps}}%
    \put(0.27787426,0.19848162){\color[rgb]{0,0,0}\makebox(0,0)[lb]{\smash{$\psi$}}}%
    \put(0.59544485,0.71453382){\color[rgb]{0,0,0}\makebox(0,0)[lb]{\smash{$v$}}}%
  \end{picture}%
\endgroup%

%% file: upperbound.eps_tex
\begingroup%
  \makeatletter%
  \providecommand\color[2][]{%
    \errmessage{(Inkscape) Color is used for the text in Inkscape, but the package 'color.sty' is not loaded}%
    \renewcommand\color[2][]{}%
  }%
  \providecommand\transparent[1]{%
    \errmessage{(Inkscape) Transparency is used (non-zero) for the text in Inkscape, but the package 'transparent.sty' is not loaded}%
    \renewcommand\transparent[1]{}%
  }%
  \providecommand\rotatebox[2]{#2}%
  \ifx\svgwidth\undefined%
    \setlength{\unitlength}{600bp}%
    \ifx\svgscale\undefined%
      \relax%
    \else%
      \setlength{\unitlength}{\unitlength * \real{\svgscale}}%
    \fi%
  \else%
    \setlength{\unitlength}{\svgwidth}%
  \fi%
  \global\let\svgwidth\undefined%
  \global\let\svgscale\undefined%
  \makeatother%
  \begin{picture}(1,0.34666667)%
    \put(0,0){\includegraphics[width=\unitlength]{upperbound.eps}}%
    \put(0.16,0.29333333){\color[rgb]{0,0,0}\makebox(0,0)[lb]{\smash{$\varepsilon^{2/3}$}}}%
    \put(0.93333333,0.28666667){\color[rgb]{0,0,0}\makebox(0,0)[rb]{\smash{$\varepsilon^{2/3}$}}}%
    \put(0.00666667,0.14666667){\color[rgb]{0,0,0}\makebox(0,0)[lb]{\smash{$O(\varepsilon^{2/3})$}}}%
    \put(0.31333333,0.14){\color[rgb]{0,0,0}\makebox(0,0)[lb]{\smash{$\varepsilon$}}}%
    \put(0.84,0.09333333){\color[rgb]{0,0,0}\makebox(0,0)[rb]{\smash{$\varepsilon$}}}%
    \put(0.26666667,0.07333333){\color[rgb]{0,0,0}\makebox(0,0)[lb]{\smash{$\theta_\varepsilon$}}}%
    \put(0.90666667,0.12666667){\color[rgb]{0,0,0}\makebox(0,0)[rb]{\smash{$\theta'_\varepsilon$}}}%
    \put(0.91333333,0.23333333){\color[rgb]{0,0,0}\makebox(0,0)[rb]{\smash{$\theta'$}}}%
    \put(0.15333333,0.23333333){\color[rgb]{0,0,0}\makebox(0,0)[lb]{\smash{$\theta$}}}%
  \end{picture}%
\endgroup%